\DeclareRobustCommand{\SkipTocEntry}[4]{}
\makeindex

\documentclass[letterpaper,12pt,reqno,latexsym,amsbsy,xypic,mathrsfs,verbatim]{amsart}
\usepackage{amsfonts,amssymb,amsthm}
\usepackage{amsmath,amscd}
\usepackage{pstricks}
\usepackage{pstricks,pst-node}
\usepackage{mathrsfs}
\usepackage[all]{xy}
\usepackage[enableskew,vcentermath, stdtext]{youngtab}

\renewcommand{\subsection}[1]{\vspace{.18in}\par\noindent\addtocounter{subsection}{1}\setcounter{equation}{0}{\bf\thesubsection.\hspace{5pt}#1}}

\newtheorem{theorem}{Theorem}[section]
\numberwithin{theorem}{section}
\numberwithin{equation}{theorem}

\theoremstyle{definition}

\newtheorem{rem}[theorem]{Remark}

\theoremstyle{plain}
\newtheorem{prop}[theorem]{Proposition}
\newtheorem{thm}[theorem]{Theorem}
\newtheorem{lem}[theorem]{Lemma}
\newtheorem{cor}[theorem]{Corollary}

\newcommand{\C}{\mathbb C}
\newcommand{\K}{\mathbb K}
\newcommand{\E}{\mathbb E}
\newcommand{\Z}{\mathbb Z}
\newcommand{\N}{\mathbb N}

\newcommand{\mf}{\mathfrak}

\newcommand{\HR}{{\mathcal{H}_{r,R}}}
\newcommand{\HC}{\mathcal{H}^{\text{\sf c}}_{r}}
\newcommand{\HCR}{\mathcal{H}^{\text{\sf c}}_{r,R}}

\newcommand{\udi}{{\underline{i}}}

\newcommand{\al}{\alpha}

\newcommand{\ep}{\epsilon}
\newcommand{\la}{\lambda}

\newcommand{\Ga}{\Gamma}
\newcommand{\La}{\Lambda}

\newcommand{\F}{\mathbb F}
\newcommand{\mc}{\mathcal}

\newcommand{\sH}{\mathcal{H}}
\newcommand{\sT}{\mathcal{T}}
\newcommand{\fS}{\mathfrak{S}}

\def\sfT{{\mathsf T}}
\def\sfS{{\mathsf S}}
\def\sfR{{\mathsf R}}
\def\sfU{\mathsf U}

\newcommand{\qQnr}{\mathcal{Q}_\bsq(n,r)}
\newcommand{\vQnr}{\mathcal{Q}_\bsv(n,r)}
\newcommand{\qSnr}{\mathcal{S}_\bsq(n,r)}

\newcommand{\End}{{\rm End}}

\newcommand{\mcZ}{{\mathcal{Z}}}

\def\sQ{{\mathcal Q}}
\def\sA{{\mathcal A}}
\def\sZ{{\mathcal Z}}

\def\sfc{{\mathsf c}}

\def\bsv{{\boldsymbol v}}
\def\bsq{{\boldsymbol q}}

\def\vQr{{{\mc Q}_\bsv(n,r)}}
\def\Innr{{I(n|n,r)}}
\def\qq{{\text{queer}} $q$}
{\vskip-\lastskip\medskip
  \noindent
  {\em #1.}\enspace
  }%
{\qed\par\medskip
  }

\begin{document}

\title[The queer $q$-Schur superalgebra]{The queer $q$-Schur superalgebra}
\author[Du and Wan]{Jie Du and Jinkui Wan$^0$}\footnotetext{Corresponding author.}
\thanks{Supported by ARC DP-120101436 and NSFC-11571036.}

\address{(Du) School of Mathematics and Statistics,
University of New South Wales,
UNSW Sydney 2052,
Australia.
 }\email{j.du@unsw.edu.au}
\address{(Wan) School of Mathematics and Statistics, Beijing Institute of Technology,
Beijing, 100081, P.R. China. } \email{wjk302@hotmail.com}

\begin{abstract} As a natural generalisation of $q$-Schur algebras associated with the Hecke algebra $\HR$ (of the symmetric group), we introduce the \qq-Schur superalgebra associated with the Hecke-Clifford superalgebra $\HCR$, which, by definition, is
the endomorphism algebra of  the induced $\HCR$-module from certain $q$-permutation modules over $\HR$. We will describe certain integral bases for these superalgebras in terms of matrices and will establish the base change property for them.
 We will also identify the \qq-Schur superalgebras with the quantum queer Schur superalgebras investigated in the context of quantum queer supergroups
and provide a constructible classification of their simple polynomial
representations over a certain extension of the field $\mathbb C(\bsv)$ of complex rational functions.
\end{abstract}


\maketitle

\begin{center}
{\it We dedicate the paper to the memory of Professor James A. Green}
\end{center}

\section{Introduction}

In the determination of the polynomial representations of the complex general linear group $GL_n(\mathbb C)$, Issai Schur introduced certain finite dimensional algebras to describe the homogeneous components of these representations and to set up a correspondence between the representations of $GL_n(\mathbb C)$ of a fixed homogeneous degree $r$ and the representations of the symmetric group $\fS_r$ on $r$ letters. This correspondence forms
part of the well-known Schur--Weyl duality. J.A. Green publicised this algebraic approach in his 1980 monograph \cite{Gr} and presented the theory ``in some ways very much in keeping with the present-day ideas on representations of algebraic groups''. In particular, he used the term {\it Schur algebras} for these finite dimensional algebras and provided some standard combinatorial treatment for the structure and representations of Schur algebras.

Since the introduction of quantum groups in mid 1980s, the family of Schur algebras has been enlarged to include $q$-Schur algebras, affine $q$-Schur algebras, and their various super analogous. See, e.g.,  \cite{DJ2}, \cite{DDPW}, \cite{DDF}, \cite{DR} and the references therein. Thus, the treatment in \cite{Gr} has also been ``quantised'' in various forms which provide applications to both structure and representations of the associated quantum linear groups, finite general linear groups, general linear $p$-adic groups, general linear supergroups and so on. In this paper, we will further generalise Green's standard treatment to a new class of Schur superalgebras, the \qq-Schur superalgebra.


Originally arising from the classification of complex simple (associative) superalgebras, there are two types of $q$-Schur superalgebras. The ``type $\tt M$'' $q$-Schur superalgebra links representations of quantum linear supergroups with those of Hecke algebras of type $A$, while the ``type $\tt Q$'' $q$-Schur superalgebra, or simply the \qq-Schur superalgebra, links the quantum queer supergroup with the Hecke--Clifford superalgebra. It is known that the former has various structures shared with the $q$-Schur algebra; see \cite{Mi} , \cite{DR} , \cite{ElK} and \cite{DG}, while the latter has been introduced in \cite{Ol,DW} in the context of Schur--Weyl--Sergeev duality and investigated in \cite{DW} in terms of a Drinfeld--Jimbo type presentation (cf. \cite{ElK}).

In this paper, we will explore the algebraic properties of the \qq-Schur superalgebra.
Thus, like the $q$-Schur algebra, we first introduce a \qq-Schur superalgebra
as the endomorphism algebra of a direct sum of certain ``$q$-permutation supermodules'' for the Hecke--Clifford superalgebra. We then construct its standard basis arising from analysing double cosets in the symmetric group. In particular, we use matrices to label the basis elements. In this way, we establish the base change property for the \qq-Schur superalgebra. We then identify them in the generic case with the quotients of the quantum queer supergroups and classify their irreducible representations in this case.


Here is the layout of the paper. In  \S \ref{sec:QqSchur}, we review the definition of the Hecke--Clifford superalgebra $\HCR$ over a commutative ring $R$ and its standard basis. We then introduce certain (induced) $q$-permutation supermodules and define the \qq-Schur superalgebra $\mc Q_q(n,r;R)$.
In  \S \ref{sec:permutation supermodules}, we study the properties of $q$-permutations supermodules and investigate some new bases. In  \S \ref{element}, we introduce some special elements in the Clifford superalgebras which is the key to a basis for the endomorphism algebra of the induced $q$-trivial representation.  An integral basis and the base change property for the \qq-Schur superalgebra are given in  \S \ref{sec:integral basis}. We then identify the \qq-Schur superalgebras as quotients of the quantum queer supergroup in \S6. Using the identification together with a work by Jones and Nazarov \cite{JN}, a complete set of irreducible polynomial (super)representations of the quantum queer supergroup is constructed in the last section.

{\it Throughout the paper}, let $\Z_2=\{0,1\}$. We will use a twofold meaning for $\Z_2$.
We will regard $\Z_2$ as an abelian group when we use it to describe a superspace.
However,  for a matrix or an $n$-tuple with entries in $\Z_2$, we will regard it as a subset of $\Z$.

Let $\sA=\mathbb Z[\bsq]$ and let  $R$ be a commutative ring which is also an $\sA$-module by mapping $\bsq$ to $q\in R$. We assume that the characteristic of $R$ is not equal to 2. If $\sQ$ denotes an $\sA$-algebra or an $\sA$-module,
then $\sQ_R=\sQ\otimes_{\sA} R$ denotes the object obtained by base change to $R$. From \S6 on wards, we will use the
 ring $\mcZ=\mathbb Z[\bsv,\bsv^{-1}]$ of Laurent polynomials in indeterminate $\bsv$, where $\bsq=\bsv^2$.

\vspace{.3cm}

\noindent
{\bf Acknowledgement.} The research was carried out while Wan was visiting
the University of New South Wales during the year 2012--2013. The hospitality and support of UNSW are gratefully acknowledged.


\section{Hecke-Clifford superalgebras and \qq-Schur superalgebras}\label{sec:QqSchur}
We first introduce an endomorphism algebra of an induced module from certain ``$q$-permutation modules'' over the Hecke--Clifford superalgebra. We will prove in \S6 that it is isomorphic to the quantum queer Schur superalgebra defined in \cite{DW} by studying the tensor space of the natural representation of the quantum queer supergroup $U_{\bsv}(\mf q_n)$.

Let $R$ be a commutative ring of characteristic not equal to 2 and let $q\in R$.
Let $\mc C_r$ denote the {\it Clifford superalgebra} over $R$ generated by odd elements $c_1,\ldots,c_r$
subject to the relations
\begin{equation}\label{Cl}
c_i^2=-1, \quad c_ic_j=-c_jc_i, \quad 1\leq i\neq j\leq r.
\end{equation}
The {\it Hecke-Clifford superalgebra} $\HCR$ is the
associative superalgebra over the ring $R$ with the
even generators $T_1,\ldots,T_{r-1}$ and the odd generators
$c_1,\ldots,c_r$ subject to \eqref{Cl} and the following additional relations: for $1\leq i,i'\leq r-1$ and $1\leq j\leq r$ with $|i-i'|>1$ and $j\neq i, i+1$,
\begin{equation}\label{commute-Tc}
\aligned
(T_i-q)&(T_i+1)=0,  \quad T_iT_{i'}=T_{i'}T_i, \quad T_iT_{i+1}T_i =T_{i+1}T_iT_{i+1},\\
T_ic_j &=c_jT_i,  \quad T_ic_i=c_{i+1}T_i,  \quad
T_ic_{i+1}=c_iT_i-(q-1)(c_i-c_{i+1}).
  \endaligned
\end{equation}
Note that if $q$ is invertible then $T_i^{-1}$ exists and the last relation is obtained by taking the inverses of the last second relation.

For notational simplicity,  if $R=\sA:=\mathbb Z[\bsq]$, the integral polynomial ring in $\bsq$, then we write $\HC=\mathcal{H}^{\text{\sf c}}_{r,\sA}$.

 Let $\mf S_r$ be the symmetric group on $\{1,2,\ldots,r\}$ which is generated by the simple transpositions $s_k=(k,k+1)$
for $1\leq k\leq r-1$. The subalgebra $\mc H_{r,R}$ of $\HCR$ generated by $T_1,\ldots,T_{r-1}$
is the (Iwahori-)Hecke algebra associated with $\mf S_r$.

 \begin{rem}\label{JN}
  Jones and Nazarov \cite{JN} introduced the notion of an affine Hecke-Clifford superalgebra which is generated by $T_1,\ldots,T_{r-1},c_1,\ldots,c_r,X_1,\ldots,X_r$ subject to the relations~\eqref{Cl}, \eqref{commute-Tc} and certain additional explicit relations which are not mentioned here.
It is known from \cite[Proposition 3.5]{JN} (cf. also \cite[Remark 4.2]{BK}) that $\HCR$ can naturally be viewed as a quotient superalgebra of the  affine Hecke-Clifford superalgebra by the two-sided ideal generated by $X_1$.

We also note that the generator $T_i$ used in \cite{BK} is equal to $v^{-1}T_i$ here, where $v^2=q$ for some $v\in R$. Moreover, we choose $c_j^2=-1$ in
\eqref{Cl}, following \cite{Ol,Se}, whereas in {\em loc. cit.} the authors take $C_j^2=1$. Thus, assuming $\sqrt{-1}\in R$, we have $c_i=\sqrt{-1}C_i$.
\end{rem}

 Like the Hecke algebra $\HR$, the symmetry of $\HCR$ can be seen from the following (anti-)automorphisms of order 2.

\begin{lem}\label{anti-inv} The $R$-algebra $\HCR$ admits the following algebra
involutions $\varphi,\psi$ and anti-involutions $\tau,\iota$ defined by
\begin{enumerate}
\item $\varphi:\, T_i\mapsto -T_{r-i}+(q-1),\quad c_j\mapsto c_{r+1-j}$;
\item $\psi: T_i\mapsto c_iT_ic_{i+1},\quad c_j\mapsto c_j$;
\item $\tau:\,T_i\mapsto -T_i+(q-1),\quad c_j\mapsto c_j$;
\item $\iota:\,T_i\mapsto T_i-(q-1)c_ic_{i+1},\quad c_j\mapsto c_j$,
\end{enumerate}
 for all  $1\leq i,j\leq r$ and $i\neq r$,
\end{lem}
\begin{proof}  Note that the correspondence $T_i\mapsto -T_i+(q-1)$ in (3), denoted by $(\;\;)^\#$ in \cite{DJ2}, defines an automorphism of the subalgebra $\HR$. So (3) follows easily.  (If $q^{-1}\in R$, then (3) is seen easily by the fact $-T_i+(q-1)=-qT_i^{-1}$.)
By directly checking the relations \eqref{Cl} and \eqref{commute-Tc}, there is an anti-involution
\begin{equation*}
\gamma:\, T_i\longmapsto T_{r-i},\;\; c_j\longmapsto c_{r+1-j}.
\end{equation*}
Thus, compositing $\gamma$ and $\tau$ gives (1).

Let $\dot T_i=c_iT_ic_{i+1}$. Then $\dot T_i, c_j$ generate $\HCR$ and satisfy all the relations \eqref{Cl} and \eqref{commute-Tc} above with $T_i$ replaced by $\dot T_i$. Hence, $\psi$ in (2) is an automorphism of order 2.
 Finally, the map $\iota$ in (4) is the map $\psi$ followed by the map $\tau$ defined relative to the generators $\dot T_i,c_j$, since, by the last relation in \eqref{commute-Tc},
$$\iota(T_i)=T_i-(q-1)c_ic_{i+1} =-c_iT_ic_{i+1}+(q-1).$$
(Of course, by noting Remark \ref{JN}, assertion (4)
follows also from \cite[(2.31)]{BK}.)
\end{proof}

Let $T_{s_i}=T_i$ and $T_w=T_{s_{i_1}}\cdots T_{s_{i_k}}$ where $w=s_{i_1}\cdots s_{i_k}\in\mf S_r$ is a reduced expression.
Let $\leq$ be the Bruhat order of $\mf S_r$, that is,  $\pi\leq w$ if $\pi$ is a subexpression of some reduced expression for $w$.
For $\al=(\al_1,\ldots,\al_r)\in\Z_2^r$, set $c^{\al}=c_1^{\al_1}\cdots c_r^{\al_r}$.
Observe that the symmetric group $\mf S_r$ acts naturally on $\mc C_r$ by permutating
the generators $c_1,\ldots,c_r$. We shall denote this action by $z\mapsto w\cdot z$
for $z\in\mc C_r, w\in\mf S_r$. Then we have
$$
w\cdot c^\al=c^{\al_1}_{w(1)}\cdots c^{\al_r}_{w(r)}=\pm c^{\al\cdot w^{-1}},
$$
where $\al\cdot w=(\al_{w(1)},\ldots,\al_{w(r)})$,  for $\al\in\Z_2^r, w\in\mf S_r$, is the place permutation action.
Moreover, by~\eqref{commute-Tc}, we have
\begin{equation}\label{ccT}
c^{\al_{k}}_{k}c^{\al_{k+1}}_{k+1}T_{s_k}=T_{s_k}c^{\al_{k}}_{k+1}c^{\al_{k+1}}_{k}+(q-1)(c_k^{\al_k+\al_{k+1}}-c^{\al_k}_{k+1}c^{\al_{k+1}}_{k})
\end{equation}
Thus, by induction on the length $\ell(w)$ of $w$, one proves
the following formula: for any $\al\in\Z_2^r$ and $w\in\mf S_r$
\begin{equation}\label{HC-PBW}
c^\al T_w=T_w(w^{-1}\cdot c^\al)+\sum_{\pi<w,\beta\in\Z_2^r}a^{\al,w}_{\pi,\beta}T_\pi c^{\beta}
=\pm T_w c^{\al\cdot w}+\sum_{\pi<w,\beta\in\Z_2^r}a^{\al,w}_{\pi,\beta}T_\pi c^{\beta},
\end{equation}
for some $a^{\al,w}_{\pi,\beta}\in R$.
Then, by the results \cite[Proposition 2.1]{JN} and \cite[Theorem 3.6]{BK} over fields,  it is easy to deduce that the following holds.

\begin{lem}\label{stdbs}
The superalgebra $\HCR$ is a free $R$-module. Moreover, both the sets $\mathscr B=\{T_w c^\al\mid w\in\mf S_r,\al\in\Z_2^r\}$
and $\{c^\al T_w \mid w\in\mf S_r,\al\in\Z_2^r\}$  form bases for $\HCR$.
\end{lem}
\begin{proof} By \eqref{HC-PBW}, it suffices to prove that $\mathscr B$ forms a basis. First, the set spans $\HCR$ by \eqref{HC-PBW}. Suppose next $R=\sA$. By base change to the fraction field of $\sA$ and \cite[Theorem 3.6]{BK}, we see that $\mathscr B$ is linearly independent and, hence, forms an $\sA$-basis for $\HC$. For a general $R$, specialization to $R$ by sending $\bsq$ to $q$ gives a basis $\{T_wc^\al\otimes 1\mid w\in\mf S_r,\al\in\Z_2^r\}$ for $\HC\otimes_\sA R$. Now, the presentation of $\HCR$ gives an $R$-algebra epimorphism $f:\HCR\to\HC\otimes_\sA R$, sending $T_i,c_j$ to $T_i\otimes 1,c_j\otimes 1$, respectively, and sending $\mathscr B$ to a basis. This forces that $\mathscr B$ itself must be a basis (and that $f$ is an isomorphism).
\end{proof}

\begin{cor}\label{HC-PBW2} Let $R$ be a commutative ring which is an $\sA$-module via the map $\sA\to R, \bsq\mapsto q$.
\begin{enumerate}
\item There is an algebra isomorphism $\HCR\cong\HC\otimes_{\sA}R$.

\item As a module over the subalgebra $\HR$, $\HCR$ is $\HR$-free with basis $\{c^\al\mid\al\in\Z_2^r\}$.
\end{enumerate}
\end{cor}

An $l$-tuple $(\la_1,\ldots,\la_l)$ of nonnegative integers $\la_i\in\mathbb N$ is called a {\it composition} of $r$ if $\sum_{i=1}^l\la_i=r$.
For a composition $\la=(\la_1,\ldots,\la_l)$ of $r$, denote by $\mf S_\la=\mf S_{\la_1}\times\cdots\times\mf S_{\la_l}$ the standard Young subgroup of $\mf S_r$ and let $\mc D_\la$ be the set of minimal length right coset representatives of $\mf S_\la$ in $\mf S_r$.
Let $\mc H_{\la,R}$ and $\mc H_{\la,R}^{\sf c}$ be the associated subalgebras of $\mc H_{r,R}$ and $\HCR$, respectively.
Then clearly $\mc H_{\la,R}^{\sf c}\cong\mc H_{\la,R}\otimes\mc C_r$ as $R$-supermodules and moreover we have algebra isomorphisms
\begin{equation}\label{Parabolic-subalgebra}
\aligned
\mc H_{\la,R}\cong \mc H_{\la_1,R}\otimes\cdots\otimes\mc H_{\la_l,R}, \quad
\mc H^{\mathsf c}_{\la,R}\cong \mc H^{\mathsf c}_{\la_1,R}\otimes\cdots\otimes\mc H^{\mathsf c}_{\la_l,R}.
\endaligned
\end{equation}
We remind the reader that, for two superalgebras $\mc S$ and $\mc T$ over $R$, the (super) tensor product $\mc S\otimes\mc T$ is naturally a superalgebra, with multiplication defined by
$$
(s\otimes t)(s'\otimes t')=(-1)^{\widehat{t}\cdot\widehat{s'}}(ss')\otimes (tt') \quad \text{for all }s,s'\in\mc S,t,t'\in\mc T,
$$
where $t,s'$ are homogeneous. Here we used the notation $\widehat{a}=0$ if $a$ is even and $\widehat{a}=1$ otherwise for a homogeneous element $a$ in a superspace.

Denote by $\Lambda(n,r)\subset \mathbb N^n$ the set of compositions of $r$ with $n$ parts.
Given $\la\in\La(n,r)$, set
\begin{equation}\label{xla}
x_\la=\sum_{w\in\mf S_\la}T_{w},\quad
y_{\la}=\sum_{w\in\mf S_\la}(-q^{-1})^{\ell(w)}T_{w},
\end{equation}
where $\ell(w)$ is the length of $w$.
As a super analog of the $q$-Schur algebra (cf. \cite[2.9]{DJ2}) or a quantum analog of the Schur superalgebra of type $\tt Q$ (cf. \cite{BK}), we introduce the {\it \qq-Schur superalgebra}:
\begin{equation}\label{QqSchur-defn}
\mc Q_q(n,r;R)={\rm End}_{\HCR}\bigg(\bigoplus_{\la\in\La(n,r)}x_\la\HCR\bigg).
\end{equation}
In particular, we write
\begin{equation}\label{v-Schur}
\mc Q_\bsq(n,r):=\mc Q_\bsq(n,r;\sA)\text{ and } \mc Q_\bsv(n,r):=\mc Q_\bsq(n,r;\mc Z),
\end{equation}
where $\bsq=\bsv^2$.

Note that $\bigoplus_{\la\in\La(n,r)}x_\la\HR$ is a direct sum of $q$-permutation modules, which is known as the tensor space used in defining $q$-Schur algebras. Since $x_\la\HCR\cong x_\la\HR\otimes_{\HR}\HCR$ as $\HCR$-supermodules, which will be called a {\it $q$-permutation supermodule}, it is clear we have $\HCR$-module isomorphism
\begin{equation}\label{ind}
\bigoplus_{\la\in\La(n,r)}x_\la\HCR\cong \bigg(\bigoplus_{\la\in\La(n,r)}x_\la\HR\bigg)\otimes_{\HR}\HCR,
\end{equation}
which is an induced $\HCR$-module from the tensor space. We will see in \S6 that this module itself can be regarded as a tensor product of a free $R$-supermodule.

Given a right $\HCR$-supermodule $M$, we can twist the action with $\varphi$ given in Lemma~\ref{anti-inv} to get a new $\HCR$-supermodule $M^{\varphi}$. Thus, as an $R$-module $M^{\varphi}=M$, the new $\HCR$-action is defined by
$$m\cdot h=m\varphi(h),\quad\text{ for all }\quad h\in\HCR,m\in M.$$
Clearly $\mc H_{r,R}$ is stable under the involution $\varphi$.
If $M$ is a $\mc H_{r,R}$-module, we can define $M^{\varphi}$ similarly.
It is known that
$$(x_\la\mc H_{r,R})^\varphi\cong y_{\la^\circ}\mc H_{r,R}\cong y_\la\mc H_{r,R}$$ (see \cite[(2.1)]{DJ2}, or \cite[Lemma 7.39]{DDPW}), where $\la^\circ=(\la_l,\la_{l-1},\ldots,\la_1)$ if $\la=(\la_1,\ldots,\la_l)$.
This implies $(x_\la \HCR)^{\varphi}\cong y_\la\HCR$ due to the facts $x_\la \HCR\cong x_\la\mc H_{r,R}\otimes_{\mc H_{r,R}}\HCR$
and $y_\la \HCR\cong y_\la\mc H_{r,R}\otimes_{\mc H_{r,R}}\HCR$.
Hence, we obtain a super analog of \cite[Theorem~2.24(i),(ii)]{DJ2}
\begin{equation}\label{y-version}
\mc Q_q(n,r;R)\cong{\rm End}_{\HCR}\bigg(\bigoplus_{\la\in\La(n,r)}y_\la\HCR\bigg).
\end{equation}

\begin{rem}\label{tau twist} We may also use the anti-involution $\tau$ in Lemma \ref{anti-inv}(3) to turn a left $\HCR$-supermodule to a right $\HCR$-supermodule or vice versa.
More precisely, if $M$ is a finite dimensional left $\HCR$-supermodule, then we define a right $\HCR$-supermodule
$M^{\tau}$ by setting $M^\tau=M$ as a $R$-module and making $\HCR$-action through $\tau$:
$$
m\cdot h=\tau(h)m,\quad \forall h\in\HCR, m\in M.
$$
In particular, if $M=\HCR y_\la$, then
\begin{equation}\label{iota-twist}
(\HCR y_\la)^\tau\cong \tau(\HCR y_\la)=\tau(y_\la)\HCR= x_\la\HCR
\end{equation}
by \cite[(2.1)]{DJ2} again.

\end{rem}

%
\section{$q$-Permutation supermodules for $\HCR$}\label{sec:permutation supermodules}

The $q$-permutation supermodules $x_\la\HCR$ share certain nice properties with $q$-permutations modules $x_\la\mc H_{r,R}$ of the Hecke algebra $\mc H_{r,R}$ as we will see below.

Recall, for a composition $\la$ of $r$, $\mc D_\la$ is the set of minimal length
right coset representatives of $\mf S_{\la}$ in $\mf S_r$.
\begin{lem} \label{ylaHC}
Let $\la$ be a composition of $r$. Then:--
\begin{enumerate}
\item the right $\HCR$-supermodule $x_{\la}\HCR$ (resp., $y_{\la}\HCR$) is $R$-free with basis
$$\{x_\la T_dc^\al\mid d\in\mc D_\la,\al\in\Z_2^r\}\quad\text{resp., }\{y_\la T_dc^\al\mid d\in\mc D_\la,\al\in\Z_2^r\};$$

\item $x_{\la}\HCR=\{h\in\HCR\mid T_{s_k}h=qh, \forall s_k\in\mf S_\la\}$.
\item $y_{\la}\HCR=\{h\in\HCR\mid T_{s_k}h=-h, \forall s_k\in\mf S_\la\}$.
\end{enumerate}
Similar statements hold for the left $\HCR$-supermodules $\HCR x_{\la}$ and $\HCR y_{\la}$.
\end{lem}
\begin{proof}
Part (1) is clear. We now prove part (2). Since $T_{s_k}x_\la=qx_\la$ for all $s_k\in\mf S_\la$,
we obtain $x_{\la}\HCR\subseteq\{h\in\HCR\mid T_{s_k}h=qh, \forall s_k\in\mf S_\la\}=:\mc E$.

Conversely, given $h=\sum_{w\in\mf S_r, \al\in\Z_2^r}h_{w,\al}T_wc^\al\in\mc E$ with $h_{w,\al}\in R$,
write $$h_\al=\sum_{w\in\mf S_r}h_{w,\al}T_wc^\al=\bigg(\sum_{w\in\mf S_r}h_{w,\al}T_w\bigg)c^\al.$$
Then $h_\al\in\mc H_{r,R}c^\al$ and $h$ can be written as $h=\sum_{\al\in\Z_2^r}h_\al$.
Assume $T_{s_k}h=qh$ for all $s_k\in\mf S_\la$. Then we obtain
$$
\sum_{\al\in\Z_2^r} (T_{s_k}h_\al-qh_\al)=0
$$
for $s_k\in\mf S_\la$. Observe that $T_{s_k}h_\al-qh_\al\in\mc H_{r,R}c^\al$ for each $\al\in\Z_2^r$.
Then, by Corollary~\ref{HC-PBW2}(2), one may deduce that
$
T_{s_k}h_\al=qh_\al
$ for all $s_k\in\mf S_\la$ and $\al\in\Z_2^r$.
This implies $T_{s_k}(\sum_{w\in\mf S_r}h_{w,\al}T_w)=q(\sum_{w\in\mf S_r}h_{w,\al}T_w)$ for all $s_k\in\mf S_\la$
since $c^\al$ is invertible.
Then using the classical result for $\mc H_{r,R}$, we obtain $\sum_{w\in\mf S_r}h_{w,\al}T_w\in x_\la\mc H_{r,R}$.
This means $h_\al\in x_\la\HCR$ for $\al\in\Z_2^r$.
Therefore, part (2) is proved. The proof of (3) is similar.

Applying the anti-involution $\tau$ in Lemma~\ref{anti-inv}(3) to (1)--(3) gives the last assertion.
\end{proof}

\begin{lem}\label{hom-interc}
The following isomorphism of $R$-modules holds for compositions $\la, \mu$ of $r$:
$$
{\rm Hom}_{\HCR}(x_\mu\HCR,x_\la\HCR)\cong x_\la\HCR\cap\HCR x_\mu.
$$
\end{lem}
\begin{proof} The proof is standard with the required isomorphism give by the map
\begin{align*}
 {\rm Hom}_{\HCR}(x_\mu\HCR,x_\la\HCR)\longrightarrow x_\la\HCR\cap\HCR x_\mu,\quad f\longmapsto f(x_\mu).
\end{align*}
\end{proof}


We will prove that the $R$-modules $x_\la\HCR\cap\HCR x_\mu$ is free for the special case $\la=\mu=(r)$ in \S4 and for the general case in \S5.
We need some preparation. In the following, we will display a new basis for $\HCR$ and $x_\la\HCR$.
We start with recalling several facts for Hecke algebras $\HR$.

Given two compositions $\la,\mu$ of $r$, let $\mc D_{\la,\mu}=\mc D_\la\cap\mc D^{-1}_\mu$, where $\mc D^{-1}_\mu=\{d^{-1}\mid d\in\mc D_\mu\}$ is the set of minimal length
left coset representatives of $\mf S_{\mu}$ in $\mf S_r$.
\begin{lem}[{\cite{Car}, \cite[Lemma 1.6]{DJ1}}]\label{DJ}
Suppose $\la,\mu$ are compositions of $r$.
\begin{enumerate}
\item The set $\mc D_{\la,\mu}$ is a system of $\mf S_\la$-$\mf S_\mu$ double coset representatives in $\mf S_r$.

\item The element $d\in\mc D_{\la,\mu}$ is the unique element of minimal length in $\mf S_\la d\mf S_\mu$.

\item If $d\in\mc D_{\la,\mu}$, there exists a composition $\nu(d)$ of $r$ such that $\fS_{\nu(d)}=d^{-1}\mf S_\la d\cap \mf S_\mu$.

\item Each element $w\in\mf S_\la d\mf S_\mu$ can be written uniquely as $w=ud\sigma$ with $u\in\mf S_\la$ and $\sigma\in\mc D_{\nu(d)}\cap \mf S_\mu$.
Moreover, $\ell(w)=\ell(u)+\ell(d)+\ell(\sigma)$.
\end{enumerate}
\end{lem}
Suppose $\la,\mu$ are compositions of $r$. Observe that $\mc D_{\mu,\la}=\mc D_\mu\cap\mc D^{-1}_\la=\mc D_{\la,\mu}^{-1}$ is a system of
$\mf S_\mu$-$\mf S_\la$ double coset representatives in $\mf S_r$ and hence $d\mf S_\mu d^{-1}\cap\mf S_\la$ is a standard Young subgroup of $\mf S_\la$
for each $d\in\mc D_{\la,\mu}$ by Lemma~\ref{DJ}(3). Thus $\nu(d^{-1})$ is well-defined via $\mf S_{\nu(d^{-1})}=d\mf S_\mu d^{-1}\cap\mf S_\la$
and moreover $\mf S_{\nu(d^{-1})}=d\mf S_{\nu(d)}d^{-1}$.
\begin{cor}\label{DJ-cor}
Suppose $\la,\mu$ are compositions of $r$ and $d\in\mc D_{\la,\mu}$.
\begin{enumerate}
\item For each $u\in\mf S_\la$, the element $ud$ can be written uniquely as
$ud=u' d\tau$ with $u'\in \mc D^{-1}_{\nu(d^{-1})}\cap \mf S_\la$ and $\tau\in\mf S_{\nu(d)}$.
Moreover $\ell(ud)=\ell(u)+\ell(d)=\ell(u')+\ell(d)+\ell(\tau)$.

\item The coset $\mf S_\la d=\{u'd\tau\mid u'\in \mc D^{-1}_{\nu(d^{-1})}\cap \mf S_\la,\tau\in\mf S_{\nu(d)}\}$.

\item Each element $w\in\mf S_\la d\mf S_\mu$ can be written uniquely as $w=u'd\pi$ with $u'\in \mc D^{-1}_{\nu(d^{-1})}\cap \mf S_\la$
 and $\pi\in\mf S_\mu$.

\item For $d_1,d_2\in\mc D_{\la,\mu}$ and $u_1\in\mc D^{-1}_{\nu(d_1^{-1})}\cap \mf S_\la, u_2\in \mc D^{-1}_{\nu(d_2^{-1})}\cap \mf S_\la$, if $u_1\neq u_2$ or $d_1\neq d_2$, then $u_1d_1\mf S_\mu\cap u_2 d_2\mf S_\mu=\emptyset.$
\end{enumerate}
\end{cor}

\begin{cor}\label{ylaTd}
Suppose $\la,\mu$ are compositions of $r$ and $d\in\mc D_{\la,\mu}$. Then
$$
x_\la T_d=\bigg(\sum_{u'\in\mc D^{-1}_{\nu(d^{-1})}\cap\mf S_\la}T_{u'}\bigg)T_d~ x_{\nu(d)}.
$$
\end{cor}

It is known that $\HR$ has the basis $\{T_w\mid w\in\fS_r\}$ satisfying the  following:
\begin{equation}\label{TwTi}
T_wT_{s_i}=\left\{
\begin{array}{ll}
T_{ws_i},&\text{ if }\ell(ws_i)=\ell(w)+1,\\
(q-1)T_{w}+q T_{ws_i}, &\text{ if }\ell(ws_i)=\ell(w)-1.
\end{array}
\right.
\end{equation}

Recall that $\leq$ is the Bruhat order of $\mf S_r$.
Then by~\eqref{TwTi} one can deduce that, for any $y,w\in \fS_r$, there is an element $y*w\in\fS_r$ such that
$\ell(y*w)\leq\ell(y)+\ell(w)$ and
\begin{equation}\label{TsigmaTw}
T_yT_w=\sum_{z\leq y*w}f^{y,w}_{z}T_{z}
\end{equation}
for some $f^{y,w}_z\in R$ with $f^{y,w}_{y*w}\neq0$ (see, e.g., \cite[Proposition 4.30]{DDPW}).


\begin{lem} \label{TuTdTw}
Suppose $\la,\mu\in\La(n,r)$. Let $u\in\mf S_\la, d\in\mc D_{\la,\mu},$ and  $w\in\mf S_\mu$.
Then we have
$$
T_uT_dT_w=\sum_{y\leq ud*w}f^{ud,w}_{y} T_{y}\text{ for some }f^{ud,w}_{y}\in R.
$$
In particular, $\ell(y)\leq \ell(u)+\ell(d)+\ell(w)$ whenever $f^{ud,w}_{y}\neq 0$.
\end{lem}
\begin{proof} Clearly from Lemma~\ref{DJ}(3) and \eqref{Twcal}, we have $T_uT_d=T_{ud}$.
Then the lemma follows from \eqref{TsigmaTw}.
\end{proof}
\begin{lem}\label{HC-basis2} For any given $\la,\mu\in\La(n,r)$, the set
$$\mathscr B'=\{T_uT_dc^\al T_\sigma\mid u\in\mf S_\la,d\in\mc D_{\la,\mu},\sigma\in\mc D_{\nu(d)}\cap\mf S_\mu, \al\in\Z_2^r\}$$ forms an $R$-basis for $\HCR$.
\end{lem}
\begin{proof}Let  $M$ be the $R$-submodule of $\HCR$ spanned by the set $\mathscr B'$.
Take an arbitrary $w\in\mf S_r$ and $\al\in\Z_2^r$.
We now prove every $T_wc^\al\in M$ by induction on $\ell(w)$.
Clearly, if $\ell(w)=0$, then $T_w=1$ and $c^\al\in M$.
Assume now $\ell(w)\geq 1$. By Lemma~\ref{DJ}(1) and (4), there exist unique $d\in\mc D_{\la,\mu}$ and $u\in\mf S_\la,\sigma\in\mc D_{\nu(d)}\cap\mf S_\mu$
such that $w=ud\sigma$ and $\ell(w)=\ell(u)+\ell(d)+\ell(\sigma)$. Then by \eqref{TwTi} we obtain $T_w=T_uT_dT_\sigma$.
Hence, by \eqref{HC-PBW} and Lemma~\ref{TuTdTw}, we have
\begin{equation}\label{Twcal}
\aligned
T_wc^\al=&~ T_uT_d(T_\sigma c^\al)=\pm T_uT_d \bigg( c^{\al\cdot\sigma^{-1}}T_{\sigma}-\sum_{\pi<\sigma,\beta\in\Z_2^r}a^{\al\cdot\sigma^{-1},\sigma}_{\beta,\pi}T_\pi c^\beta\bigg)\\
=&\pm T_uT_dc^{\al\cdot\sigma^{-1}}T_{\sigma}\mp \sum_{\pi<\sigma,\beta\in\Z_2^r}a^{\al\cdot\sigma^{-1},\sigma}_{\beta,\pi}T_uT_dT_\pi c^\beta\\
=&~ \pm T_uT_dc^{\al\cdot\sigma^{-1}}T_{\sigma}\mp\sum_{\pi<\sigma,\beta\in\Z_2^r}a^{\al\cdot\sigma^{-1},\sigma}_{\beta,\pi}\sum_{u'\leq ud*\pi, \beta\in\Z_2^r}f^{ud,\pi}_{u',\beta}T_{u'} c^\beta,
\endaligned
\end{equation}
where $u'$ satisfying $\ell(u')\leq\ell(u)+\ell(d)+\ell(\pi)$.
Since $\pi<\sigma$, we have $\ell(u')\leq\ell(u)+\ell(d)+\ell(\pi)<\ell(u)+\ell(d)+\ell(\sigma)=\ell(w)$.
Then, by induction on $\ell(w)$, each term $T_{u'} c^\beta$ on the right hand side of \eqref{Twcal} belongs to $M$ and so does $T_wc^{\al}$ by \eqref{Twcal}.
Hence, $M=\HCR$, or equivalently, $\HCR$ is spanned by the set $\mathscr B'$ by Lemma~\ref{stdbs}. Finally,
$|\mathscr B|=|\mathscr B'|$ forces  that the set is linearly independent.
\end{proof}
For $\la,\mu\in\La(n,r)$, let
$X_{\la,\mu}=\big( \bigcup_{d\in\mc D_{\la,\mu}}\big(\{d\}\times (\mc D_{\nu(d)}\cap\fS_\mu)\big)\big)\times \Z_2^r.$
\begin{cor}\label{trivial-intersect}
For given $\la,\mu\in\La(n,r)$, we may decompose $\HCR$ into  a direct sum of left $\mc H_{\la,R}$-modules or right
$\sH_{\mu,R}^{\mathsf c}$-modules:
$$\HCR=\bigoplus_{(d,\sigma,\al)\in X_{\la,\mu}}\mc H_{\la,R}T_dc^\al T_\sigma=\bigoplus_{d\in\mc D_{\la,\mu}}\bigoplus_{u\in\mc D^{-1}_{\nu(d^{-1})}\cap\fS_\la}T_uT_d\sH_{\mu,R}^{\mathsf c}.$$
In particular, the set
$$\{x_\la T_dc^\al T_\sigma\mid(d,\sigma,\al)\in X_{\la,\mu}\} \quad (\text{resp., }  \{T_uT_dc^\alpha x_\mu\mid d\in\mc D_{\la,\mu},u\in\mc D^{-1}_{\nu(d^{-1})}\cap\fS_\la,\al\in\mathbb Z_2^r\}) $$
forms a basis for $x_\la\HCR$ (resp., $\HCR x_\mu$).
\end{cor}
\begin{proof}
Clearly, $\mc H_{\la,R}T_dc^{\al}T_\sigma$ is spanned by $\{T_uT_dc^{\al}T_\sigma\mid u\in\mf S_\la\}$.
Since
$$
\{T_uT_dc^\al T_\sigma\mid u\in\mf S_\la,(d,\sigma,\al)\in X_{\la,\mu}\}=\displaystyle\dot\bigcup_{(d,\sigma,\al)\in X_{\la,\mu}}
\{T_uT_dc^\al T_\sigma\mid u\in\mf S_\la\}
$$
is a disjoint union, the assertion follows from Lemma~\ref{HC-basis2} and the Corollary~\ref{DJ-cor}.
\end{proof}

\section{The elements $c_{q,i,j}$}\label{element}

For $r\geq 1$ and $1\leq i<j\leq r$, we set
\begin{equation}\label{cijq}
c_{q,i,j}=q^{j-i}c_i+q^{j-i-1}c_{i+1}+\cdots+qc_{j-1}+c_j,\quad
c'_{q,i,j}=c_i+q c_{i+1}+\cdots+q^{j-i}c_j
\end{equation}
For simplicity, we write $c_{q,r}=c_{q,1,r}$ and $c_{q,r}'=c_{q,1,r}'$.

\begin{lem}\label{yr-crq}
Let $r\geq 1$. The following holds in $\HCR$:
$$
x_{(r)}c_{q,r}=c'_{q,r}x_{(r)}\quad\text{and}\quad y_{(r)}c'_{q,r}=c_{q,r}y_{(r)}.
$$
\end{lem}
\begin{proof}We only prove the first formula. The proof of the second one is similar.
We apply induction on $r$.
It is known that
\begin{equation}\label{yrinduction}
x_{(r)}
=x_{(r-1)}\big(1+T_{r-1}+T_{r-1}T_{r-2}+\cdots+T_{r-1}\cdots T_1\big)
\end{equation}
since $\{1,s_{r-1},s_{r-1}s_{r-2},\ldots, s_{r-1}\cdots s_2s_1\}=\mc D_{(r-1,1)}$.
Using~\eqref{commute-Tc},
a direct calculation shows for $1\leq k\leq r-1$ and $1\leq l\leq r$:
\begin{align*}
T_{r-1}\cdots T_kc_l=\left\{
\begin{array}{ll}
c_lT_{r-1}\cdots T_k,&\text{ if }1\leq l\leq k-1,\\
c_rT_{r-1}\cdots T_k, &\text{ if }l=k,\\
c_{l-1}T_{r-1}\cdots T_k-(q-1)T_{l-2}\cdots T_k(c_k-c_r)T_{r-1}\cdots T_l,&\text{ if }k+1\leq l\leq r,
\end{array}
\right.
\end{align*}
where we use the convention $T_{r-1}\cdots T_l=1$ if $l=r$ and the formula in the last case is due to
the following computation:
\begin{align*}
T_{r-1}\cdots T_kc_l=&T_{r-1}\cdots T_l\cdot (c_{l-1}T_{l-1}-(q-1)(c_{l-1}-c_l))\cdot T_{l-2}\cdots T_k\\
=&c_{l-1}T_{r-1}\cdots T_lT_{l-1}T_{l-2}\cdots T_k-(q-1)(T_{r-1}\cdots T_l)\cdot(c_{l-1}-c_l)\cdot (T_{l-2}\cdots T_k)\\
=&c_{l-1}T_{r-1}\cdots T_k-(q-1)(c_{l-1}-c_r)\cdot (T_{r-1}\cdots T_l)\cdot (T_{l-2}\cdots T_k)\\
=&c_{l-1}T_{r-1}\cdots T_k-(q-1)(c_{l-1}-c_r)\cdot (T_{l-2}\cdots T_k)\cdot (T_{r-1}\cdots T_l)\\
=&c_{l-1}T_{r-1}\cdots T_k-(q-1)(T_{l-2}\cdots \cdot T_k)\cdot(c_{k}-c_r)\cdot  (T_{r-1}\cdots T_l).
\end{align*}
This implies that the following holds for $1\leq k\leq r-1$ and $1\leq l\leq r$:
\begin{align*}
&x_{(r-1)}T_{r-1}\cdots T_kc_l\\
&=\left\{
\begin{array}{ll}
x_{(r-1)}c_lT_{r-1}\cdots T_k,&\text{ if }1\leq l\leq k-1,\\
x_{(r-1)}c_rT_{r-1}\cdots T_k, &\text{ if }l=k,\\
x_{(r-1)}c_{l-1}T_{r-1}\cdots T_k-q^{l-k-1}(q-1)x_{(r-1)}(c_k-c_r)T_{r-1}\cdots T_l,&\text{ if }k+1\leq l\leq r,
\end{array}
\right.
\end{align*}
where the formula in the last case is due to the fact $x_{(r-1)}T_{l-2}\cdots T_k=q^{l-k-1}x_{(r-1)}$.
Hence we obtain
\begin{align*}
&x_{(r-1)}T_{r-1}\cdots T_kc_{q,r}=x_{(r-1)}T_{r-1}\cdots T_k(q^{r-k+1}c_{q,1,k-1}+q^{r-k}c_k+c_{q,k+1,r})\\
&=x_{(r-1)}(q^{r-k+1}c_{q,1,k-1}+q^{r-k}c_r+c_{q,k,r-1})T_{r-1}\cdots T_k\\
&\quad -\sum^{r}_{l=k+1}q^{r-l}q^{l-k-1}(q-1)x_{(r-1)}(c_k-c_r)T_{r-1}\cdots T_l\;\;(\text{noting $T_{r-1}\cdots T_l=1$ if $l=r$})\\
&=x_{(r-1)}(q^{r-k+1}c_{q,1,k-1}+q^{r-k}c_r+c_{q,k,r-1})T_{r-1}\cdots T_k\\
&\quad -\sum^{r-1}_{l=k+1}(q-1)q^{r-k-1}x_{(r-1)}(c_k-c_r)T_{r-1}\cdots T_l-(q-1)q^{r-k-1}x_{(r-1)}(c_k-c_r).
\end{align*}
This together with~\eqref{yrinduction} gives rise to
\begin{equation}\label{yr-crq1}
\aligned
x_{(r)}c_{q,r}
&=\big(x_{(r-1)}+\sum_{k=1}^{r-1}x_{(r-1)}T_{r-1}\cdots T_k\big)c_{q,r}\\
&=x_{(r-1)}c_{q,r}+\sum^{r-1}_{k=1}x_{(r-1)}(q^{r-k+1}c_{q,1,k-1}+q^{r-k}c_r+c_{q,k,r-1})T_{r-1}\cdots T_k\\
&\quad\,-\sum^{r-1}_{k=1}\sum^{r-1}_{l=k+1}(q-1)q^{r-k-1}x_{(r-1)}(c_k-c_r)T_{r-1}\cdots T_l\\
&\quad\,-\sum^{r-1}_{k=1}(q-1)q^{r-k-1}x_{(r-1)}(c_k-c_r).
\endaligned
\end{equation}
Combining the first term and the last sum in \eqref{yr-crq1} yields
\begin{align*}
x_{(r-1)}c_{q,r}&-\sum^{r-1}_{k=1}x_{(r-1)}(q-1)q^{r-k-1}(c_k-c_r)\\
&=x_{(r-1)}\bigg(\sum^r_{k=1}q^{r-k}c_k-\sum^{r-1}_{k=1}(q^{r-k}-q^{r-k-1})c_k+\sum^{r-1}_{k=1}(q^{r-k}-q^{r-k-1})c_r\bigg)\\
&=x_{(r-1)}(q^{r-2}c_1+q^{r-3}c_2+\cdots+c_{r-1}+q^{r-1}c_r)
=x_{(r-1)}(c_{q,1,r-1}+q^{r-1}c_r).
\end{align*}
Also, the second (double) sum in \eqref{yr-crq1} can be rewritten as
\begin{align*}
&\quad\,x_{(r-1)}\sum^{r-1}_{l=2}\bigg(\sum^{l-1}_{k=1}(q-1)q^{r-k-1}(c_k-c_r)\bigg)T_{r-1}\cdots T_l\\
&=x_{(r-1)}\sum^{r-1}_{l=2}\big((q^{r-l+1}-q^{r-l})c_{q,1,l-1}-(q^{r-1}-q^{r-l})c_r\big)T_{r-1}\cdots T_l.
\end{align*}
Hence, substituting into \eqref{yr-crq1} yields
\begin{equation}\label{yr-crq2}
x_{(r)}c_{q,r}=x_{(r-1)}(c_{q,1,r-1}+q^{r-1}c_r)(1+\sum^{r-1}_{k=1}T_{r-1}\cdots T_k).
\end{equation}
Here the following calculation has been used for $2\leq k\leq r-1$:
\begin{align*}
&(q^{r-k+1}c_{q,1,k-1}+q^{r-k}c_r+c_{q,k,r-1})-\big((q^{r-k+1}-q^{r-k})c_{q,1,k-1}-(q^{r-1}-q^{r-k})c_r\big)\\
&=q^{r-k}c_{q,1,k-1}+c_{q,k,r-1}+q^{r-1}c^r=c_{q,1,r-1}+q^{r-1}c_r=c_{q,r-1}+q^{r-1}c_r,
\end{align*}
which is also the first ($k=1$) term of the first sum in \eqref{yr-crq1}.
By induction on $r$, we have
$$
x_{(r-1)}c_{q,1,r-1}=c'_{q,r-1}x_{(r-1)}
$$
and hence, 
one can deduce
\begin{align*}
x_{(r-1)}(c_{q,1,r-1}+q^{r-1}c_r)=(c'_{q,r-1}+q^{r-1}c_r)x_{(r-1)}=c'_{q,r}x_{(r-1)}.
\end{align*}
Thus, by~\eqref{yr-crq2} and \eqref{yrinduction}, we obtain
$$
x_{(r)}c_{q,r}=c'_{q,r}x_{(r-1)}(1+T_{r-1}+\cdots+T_{r-1}\cdots T_1)=c'_{q,r}x_{(r)}.
$$
This proves the lemma.
\end{proof}
%
%
\begin{prop}\label{one-part case}
For $r\geq 1$,  the intersection $x_{(r)}\HCR\cap\HCR x_{(r)}$ is a free $R$-module with basis $\{x_{(r)}, x_{(r)}c_{q,r}\}$. A similar result holds with $x_{(r)}$ replaced by $y_{(r)}$ and $c_{q,r}$ by $c_{q,r}'$.
\end{prop}
\begin{proof}
By Lemma \ref{yr-crq}, the set $\{x_{(r)}, x_{(r)}c_{q,r}\}$ is contained in the intersection
$x_{(r)}\HCR\cap\HCR x_{(r)}$ and clearly it is linearly independent.
It remains to show that $x_{(r)}\HCR\cap\HCR x_{(r)}$ is spanned by $\{x_{(r)}, x_{(r)}c_{q,r}\}$.
Take an arbitrary $z\in x_{(r)}\HCR\cap\HCR x_{(r)}$. Since $z\in x_{(r)}\HCR$, by Lemma~\ref{ylaHC}(1), we can write $z$ as
$
z=\sum_{\al\in\Z_2^r}f_\al x_{(r)}c^\al,
$
with $f_\al\in R$. Moreover, since $z\in \HCR x_{(r)}$, by Lemma~\ref{ylaHC}(2),
$zT_{s_k}=qz, \quad \forall 1\leq k\leq r-1.$

For $\al=(\al_1,\ldots,\al_r)$ and $1\leq k\leq r-1$, by \eqref{commute-Tc} and \eqref{ccT}, we have
\begin{align*}
c^{\al_{j}}_{j}T_{s_k}&=T_{s_k}c^{\al_{j}}_{j},\quad \text{ for } j\neq k,k+1,\\
c^{\al_{k+1}}_{k+1}T_{s_k}&=T_{s_k}c^{\al_{k+1}}_{k},\quad
c^{\al_{k}}_{k}T_{s_k}=T_{s_k}c^{\al_{k}}_{k+1}+(q-1)(c_k^{\al_k}-c^{\al_k}_{k+1}),\\
c^{\al_{k}}_{k}c^{\al_{k+1}}_{k+1}T_{s_k}&=T_{s_k}c^{\al_{k}}_{k+1}c^{\al_{k+1}}_{k}+(q-1)(c_k^{\al_k+\al_{k+1}}-c^{\al_k}_{k+1}c^{\al_{k+1}}_{k}).
\end{align*}
Then, by the fact $x_{(r)}T_{s_k}=qx_{(r)}$ for $1\leq k\leq r-1$, we obtain
\begin{equation}
\aligned
x_{(r)}c^{\al}~T_{s_k}=&x_{(r)}T_{s_k}c_1^{\al_1}\cdots c^{\al_{k-1}}_{k-1}c^{\al_{k}}_{k+1}c^{\al_{k+1}}_{k}c^{\al_{k+2}}_{k+2}\cdots c_r^{\al_r}\\
&+(q-1)x_{(r)}c_1^{\al_1}\cdots c^{\al_{k-1}}_{k-1}(c_k^{\al_k+\al_{k+1}}-c^{\al_k}_{k+1}c^{\al_{k+1}}_{k})c^{\al_{k+2}}_{k+2}\cdots c_r^{\al_r}\\
=&x_{(r)}c_1^{\al_1}\cdots c^{\al_{k-1}}_{k-1}c^{\al_{k}}_{k+1}c^{\al_{k+1}}_{k}c^{\al_{k+2}}_{k+2}\cdots c_r^{\al_r}\\
&+(q-1)x_{(r)}c_1^{\al_1}\cdots c^{\al_{k-1}}_{k-1}c_k^{\al_k+\al_{k+1}}c^{\al_{k+2}}_{k+2}\cdots c_r^{\al_r}.
\endaligned
\end{equation}
This together with $zT_{s_k}=qz$ implies
\begin{align*}
q\sum_{\al\in\Z_2^r}f_\al x_{(r)}c^\al=
zT_{s_k}=&\sum_{\al\in\Z_2^r}f_\al x_{(r)}c_1^{\al_1}\cdots c^{\al_{k-1}}_{k-1}c^{\al_{k}}_{k+1}c^{\al_{k+1}}_{k}c^{\al_{k+2}}_{k+2}\cdots c_r^{\al_r}\\
&+\sum_{\al\in\Z_2^r}(q-1)f_\al x_{(r)}c_1^{\al_1}\cdots c^{\al_{k-1}}_{k-1}c_k^{\al_k+\al_{k+1}}c^{\al_{k+2}}_{k+2}\cdots c_r^{\al_r}\\
\end{align*}
for all $1\leq k\leq r-1$. Then by using Lemma~\ref{ylaHC}(1) and by equating the coefficients of $x_{(r)}c^{\al}$
on both sides for each $\al\in\Z_2^r$, we obtain
\begin{equation}\label{423}
\aligned
qf_{(\al_1,\ldots,\al_{k-1},1,1,\al_{k+2},\ldots,\al_r)}
=&-f_{(\al_1,\ldots,\al_{k-1},1,1,\al_{k+2},\ldots,\al_r)},\\
qf_{(\al_1,\ldots,\al_{k-1},0,1,\al_{k+2},\ldots,\al_r)}
=&f_{(\al_1,\ldots,\al_{k-1},1,0,\al_{k+2},\ldots,\al_r)},\\
qf_{(\al_1,\ldots,\al_{k-1},0,0,\al_{k+2},\al_r)}
=&f_{(\al_1,\ldots,\al_{k-1},0,0,\al_{k+2},\ldots,\al_r)}\\
&+(q-1)(f_{(\al_1,\ldots,\al_{k-1},1,1,\al_{k+2},\ldots,\al_r)}+f_{(\al_1,\ldots,\al_{k-1},0,0,\al_{k+2},\ldots,\al_r)}),\\
qf_{(\al_1,\ldots,\al_{k-1},1,0,\al_{k+2},\ldots,\al_r)}
=&f_{(\al_1,\ldots,\al_{k-1},0,1,\al_{k+2},\ldots,\al_r)}\\
&+(q-1)(f_{(\al_1,\ldots,\al_{k-1},1,0,\al_{k+2},\ldots,\al_r)}
+f_{(\al_1,\ldots,\al_{k-1},0,1,\al_{k+2},\ldots,\al_r)}),
\endaligned
\end{equation}
where the relation \eqref{Cl} is used in the first and last equalities.
Hence, since the characteristic of $R$ is not equal to 2, \eqref{423} implies
\begin{align}\label{xxx}
f_{(\al_1,\ldots,\al_{k-1},1,0,\al_{k+1},\ldots,\al_r)}
&=qf_{(\al_1,\ldots,\al_{k-1},0,1,\al_{k+1},\ldots,\al_r)},\quad f_{(\al_1,\ldots,\al_{k-1},1,1,\al_{k+1},\ldots,\al_r)}=0
\end{align}
for all $\al_i\in\Z_2$ and $1\leq k\leq r-1$.
Then one can deduce that
\begin{align*}
\left\{
\begin{array}{ll}
f_{\ep_k}=q^{r-k}f_{\ep_r},&\text{ for }1\leq k\leq r,\\
f_{\al}=0,&\text{ for } \al\in{\mathbb Z_2^r}\backslash {\bf 0}\text{ and }\al\neq\ep_k, 1\leq k\leq r,
\end{array}
\right.
\end{align*}
where $\ep_k=(0,\ldots,0,1,0,\ldots,0)$ with $1$ in the $k$-th position.
This means
$$
z=\sum_{\al\in\Z_2^r}f_\al x_{(r)}c^\al=f_{(0,\ldots,0)}x_{(r)}+f_{\ep_r}x_{(r)}(q^{r-1}c_1+q^{r-2}c_2+\cdots+c_r)
=f_{(0,\ldots,0)}x_{(r)}+f_{\ep_r}x_{(r)}c_{q,r}.
$$
Hence $z$ lies in the $R$-module spanned by $\{x_{(r)},x_{(r)}c_{q,r}\}$ and the proposition is verified.
\end{proof}

Given $\la=(\la_1,\ldots,\la_N)\in\La(N,r)$ and $\al\in\Z_2^N$ with $N\geq 1$ and $\al_k\leq\la_k$ for $1\leq k\leq N$,
we let $\widetilde{\la}_k=\la_1+\cdots+\la_{k}$ for $1\leq k\leq N$ and introduce the following elements in $\mc C_r$:
\begin{equation}\label{cla-al}
\aligned
c^\al_\la=&(c_{q,1,\widetilde{\la}_1})^{\al_1}(c_{q,\widetilde{\la}_1+1,\widetilde{\la}_2})^{\al_2}
\cdots(c_{q,\widetilde{\la}_{N-1}+1,\widetilde{\la}_N})^{\al_N},\\
(c^\al_\la)'=&(c'_{q,1,\widetilde{\la}_1})^{\al_1}(c'_{q,\widetilde{\la}_1+1,\widetilde{\la}_2})^{\al_2}
\cdots(c'_{q,\widetilde{\la}_{N-1}+1,\widetilde{\la}_N})^{\al_N}.
\endaligned
\end{equation}
For notational simplicity, we define for $\la,\mu\in\N^N$
$$\la\leq\mu\iff\la_i\leq\mu_i\text{ for all }i=1,2,\ldots,N.$$
\begin{cor}\label{tensor-on-part-case}
Suppose $\la=(\la_1,\ldots,\la_N)\in\La(N,r)$. Then $x_{\la}\sH_{\la,R}^\sfc\cap\sH_{\la,R}^\sfc x_{\la}$ is a free $R$-module with basis
$\{x_{\la}c_\la^\al=(c^\al_\la)'x_{\la}\mid \al\in\Z^N_2, \al\leq\la\}$. A similar result holds for the $y_\la$ version.
\end{cor}
\begin{proof}
Clearly, by Lemma~\ref{yr-crq} and \eqref{Cl}, we have $x_{\la}c_\la^\al=(c^\al_\la)'x_{\la}$ for $\al\in\Z_2^N$ and the set is
linearly independent. We now modify the proof of Proposition~\ref{one-part case} to prove the assertion.
Without loss of generality, we may assume all $\lambda_i\neq 0$.
For $z=\sum_{\al\in\Z_2^r}f_{\al}x_{\la}c^{\al}\in x_{\la} \sH_{\la,R}^\sfc\cap\sH_{\la,R}^\sfc x_{\la}$,
\eqref{xxx} continues to hold for all $k$ such that $\widetilde{\la}_{i-1}+1\leq k<\widetilde{\la}_i$ for all $i$ (assuming $\la_0=0$).
Then we first observe that if there exist $\widetilde{\la}_{i-1}+1\leq k<l\leq\widetilde{\la}_i$ for some $1\leq i\leq N$ such that $\al_k=\al_l=1$, then $f_{\al}=0$. Now it remains to consider those $f_{\al}$ for $\al=\bf 0$ or for $\al\in\Z_2^r$ having the form $\al=\ep_{k_1}+\ep_{k_2}+\cdots+\ep_{k_t}$ for some $1\leq t\leq N$, where there exist $1\leq i_1<i_2<\cdots<i_t\leq N$
such that $\widetilde{\la}_{i_1-1}+1\leq k_1\leq\widetilde{\la}_{i_1}$, $\widetilde{\la}_{i_2-1}+1\leq k_2\leq\widetilde{\la}_{i_2}$,\ldots, $\widetilde{\la}_{i_t-1}+1\leq k_t\leq\widetilde{\la}_{i_t}$. In the latter case, by \eqref{xxx}, one can deduce that
$$
f_{\al}=q^{\widetilde{\la}_{i_1}-k_1}q^{\widetilde{\la}_{i_2}-k_2}\cdots q^{\widetilde{\la}_{i_t}-k_t}f_{\ep_{\widetilde{\la}_{i_1}}+\ep_{\widetilde{\la}_{i_2}}+\cdots+\ep_{\widetilde{\la}_{i_t}}}.
$$
Hence,
$$
z=f_{\bf 0}x_{\la}+\sum_{t=1}^N\sum_{1\leq i_1<\cdots<i_t\leq N}
f_{\ep_{\widetilde{\la}_{i_1}}+\ep_{\widetilde{\la}_{i_2}}+\cdots+\ep_{\widetilde{\la}_{i_t}}}x_\la
c_{q,\widetilde{\la}_{i_1-1}+1,\widetilde{\la}_{i_1}}c_{q,\widetilde{\la}_{i_2-1}+1,\widetilde{\la}_{i_2}}\cdots c_{q,\widetilde{\la}_{i_t-1}+1,\widetilde{\la}_{i_t}},
$$
 proving the spanning condition.
\end{proof}
\begin{rem}\label{ForMori1}
(1) By \eqref{Parabolic-subalgebra}, $\sH_{\la,R}^\sfc\cong\mc H^{\text{\sf c}}_{\la_1,R}\otimes\cdots\otimes\mc H^{\text{\sf c}}_{\la_N,R}$, where the isomorphism maps
 $x_\la$ to $x_{(\la_1)}\otimes\cdots\otimes x_{(\la_N)}$.  Hence, as free $R$-modules,
$$
x_\la\mc H^{\text{\sf c}}_{\la,R}\cong x_{(\la_1)}\mc H^{\text{\sf c}}_{\la_1,R}\otimes\cdots\otimes x_{(\la_N)}\mc H^{\text{\sf c}}_{\la_N,R}, \quad
\mc H^{\text{\sf c}}_{\la,R}x_\la\cong \mc H^{\text{\sf c}}_{\la_1,R}x_{(\la_1)}\otimes\cdots\otimes \mc H^{\text{\sf c}}_{\la_N,R}x_{(\la_N)}.
$$
The corollary above implies the following isomorphism of free $R$-modules
$$
x_{\la}\mc H^{\text{\sf c}}_{\la,R}\cap\mc H^{\text{\sf c}}_{\la,R} x_{\la}\cong (x_{(\la_1)}\mc H^{\text{\sf c}}_{\la_1,R}\cap\mc H^{\text{\sf c}}_{\la_1,R}x_{(\la_1)})\otimes\cdots\otimes
(x_{(\la_N)}\mc H^{\text{\sf c}}_{\la_N,R}\cap\mc H^{\text{\sf c}}_{\la_N,R}x_{(\la_N)}).
$$

(2) The elements $c_{q,i,j}$ have also been introduced in \cite[Lem.~6.7]{Mo}. They play important role in \eqref{TAB} below for the construction of an integral basis. However, such a role was not mentioned in {\it loc. cit}. Instead, they play a key role in the classification of simple modules via the superalgebras $\Gamma_\la$ (which is actually isomorphic to the endomorphism superalgebra ${\rm End}_{\HCR}(x_\la\HCR)\cong x_\la\HCR\cap\HCR x_\la$ by Corollary \ref{tensor-on-part-case}) defined in \cite[Definition~6.8]{Mo} (see \cite[Theorems 6.14\&6.32]{Mo}).
\end{rem}
\section{Integral bases and the base change property}\label{sec:integral basis}
We are now ready to prove, by constructing a basis, that the module $x_\la\HCR\cap \HCR x_\mu$ occurred in Lemma \ref{hom-interc} is $R$-free. This will provide an integral $R$-basis for $\mc Q_q(n,r;R)$.

Let $M_n(\N)$ be the set of $n\times n$-matrices $M=(m_{ij})$ and let
$$
M_n(\N)_r=\{M\in M_n(\N)\mid \sum m_{ij}=r\}.
$$
Given $M\in M_n(\N)$, define
$$
{\rm ro}(M)=(\sum_j m_{1j},\ldots,\sum_jm_{nj}),\quad
{\rm co}(M)=(\sum_j m_{j1},\ldots,\sum_jm_{jn}).
$$
Then ${\rm ro}(M), {\rm co}(M)\in\Lambda(n,r)$ for any $M\in M_n(\N)_r$.
Given a composition $\la\in\La(n,r)$ and $1\leq k\leq n$, define the subsets $R^\la_k\subseteq\{1,\ldots,r\}$ as follows:
$$
R^\la_k=\{\tilde\la_{k-1}+1,\tilde\la_{k-1}+2\ldots,\tilde\la_{k-1}+\la_k\} \quad(\tilde\la_0=0, \tilde\la_i=\la_1+\cdots+\la_i, 1\leq i\leq n).
$$
It is well known that double cosets of the symmetric group can be described in terms of matrices as follows.
There is a bijection
\begin{equation}\label{mapj}
\aligned
\j: M_n(\N)_r&\longrightarrow \mf J_n(r)=\{(\la,d,\mu)\mid \la,\mu\in\La(n,r),d\in\mc D_{\la,\mu}\}\\
M&\longmapsto ({\rm ro}(M), d_M,{\rm co}(M)),
\endaligned
\end{equation}
where $d_M\in\mc D_{\la,\mu}$ is the permutation satisfying $|R^{{\rm ro}(M)}_i\cap d_M(R^{{\rm co}(M)}_j)|=m_{ij}$ for $1\leq i,j\leq n$ with $M=(m_{ij})$.
Moreover, the composition $\nu(d_M)$ (resp. $\nu(d_M^{-1})$), defined in Lemma \ref{DJ}(3), is obtained by reading the entries in $M$ along columns  (resp. rows), that is,
\begin{equation}\label{nudA}
\aligned
\nu(d_M)=&(m_{11},\ldots,m_{n1},m_{12},\ldots,m_{n2},\ldots,m_{1n},\ldots,m_{nn}),\\
\nu(d_M^{-1})=&(m_{11},\ldots,m_{1n},m_{21},\ldots,m_{2n},\ldots,m_{n1},\ldots,m_{nn}).
\endaligned
\end{equation}

For the subset $\Z_2=\{0,1\}$ of $\N$, let $M_n(\Z_2)$ be the set of $n\times n$-matrices  $B=(b_{ij})$
with  $b_{ij}\in\Z_2$. Set
\begin{align*}
M_n(\N|\Z_2):=&\{(A|B)\mid A\in M_n(\N), B\in M_n(\Z_2)\},\\
M_n(\N|\Z_2)_r:=&\{(A|B)\in M_n(\N|\Z_2)\mid A+B\in M_n(\N)_r\}.
\end{align*}
Given $(A|B)\in M_n(\N|\Z_2)_r$ with $A=(a_{ij}), B=(b_{ij})$, let $m_{ij}=a_{ij}+b_{ij}$ for $1\leq i,j\leq n$ and then we have $A+B=(m_{ij})\in M_n(\N)_r$.
Hence by \eqref{mapj}, we have the permutation $d_{A+B}\in\mc D_{\la,\mu}$ with $\la={\rm ro}(A+B), \mu={\rm co}(A+B)$
and, moreover,
$$\nu_{A|B}:=\nu(d_{A+B})=(m_{11},\ldots,m_{n1},m_{12},\ldots,m_{n2},\ldots,m_{1n},\ldots,m_{nn}).$$
Let $\al_B=(b_{11},\ldots,b_{n1},\ldots,b_{1n},\ldots,b_{nn})\in\Z_2^{n^2}$.
Then by \eqref{cla-al} we are ready to introduce the following elements:
\begin{align}
c_{A|B}=&c^{\al_B}_{\nu_{A|B}}\in\mc C_r, \label{cAB}\\
T_{A|B}=&x_{\la}T_{d_{A+B}}c_{A|B}\sum_{\sigma\in\mathcal D_{\nu_{A|B}}\cap \mf S_{\mu}}T_{\sigma}.\label{TAB}
\end{align}
Given $\la,\mu\in\La(n,r)$ and $d\in\mc D_{\la,\mu}$, let
\begin{align*}
M_n(\N)_{\la,\mu}:=&\{A\in M_n(\N)\mid {\rm ro}(A)=\la,{\rm co}(A)=\mu\},\\
M_n(\N|\Z_2)_{\la,\mu}:=&\{(A|B)\in M_n(\N|\Z_2)\mid {\rm ro}(A+B)=\la,{\rm co}(A+B)=\mu\},\\
M_n(\N|\Z_2)^d_{\la,\mu}:=&\{(A|B)\in M_n(\N|\Z_2)_{\la,\mu}\mid d_{A+B}=d\}.
\end{align*}
Note that, by sending $A$ to $(A|0)$, we may regard $M_n(\N)_{\la,\mu}$ as a subset of $M_n(\N|\Z_2)_{\la,\mu}$. Note also that
\begin{equation}\label{yyy}
M_n(\N|\Z_2)_{\la,\mu}=\bigcup_{d\in\mc D_{\la,\mu}}M_n(\N|\Z_2)_{\la,\mu}^d.
\end{equation}
\begin{lem}\label{linear-ind}
For $\la,\mu\in\La(n,r)$, the set
$
\{T_{A|B}\mid (A|B)\in M_n(\N|\Z_2)_{\la,\mu}\}
$
is a linearly independent subset of $x_\la\HCR\cap\HCR x_\mu$.
\end{lem}
\begin{proof} We first prove that  $T_{A|B}\in x_\la\HCR\cap\HCR x_\mu$ for each $(A|B)\in M_n(\N|\Z_2)_{\la,\mu}$.
Clearly, $T_{A|B}\in x_\la\HCR$ by \eqref{TAB}.
Let $d=d_{A+B}$. Then by~\eqref{nudA}
$$
\nu(d)=(m_{11},\ldots,m_{n1},m_{12},\ldots,m_{n2},\ldots,m_{1n},\ldots,m_{nn}),
$$
where $A+B=(m_{ij})$.
Hence, by Corollary~\ref{tensor-on-part-case} and \eqref{cAB}, $x_{\nu(d)}c_{A|B}=c'_{A|B}x_{\nu(d)},$
where $c'_{A|B}=(c^{\al_B}_{\nu_{A|B}})'$.
On the other hand, the partition
\begin{equation}\label{decomp-Smu}
\mf S_\mu=\dot{\bigcup}_{\sigma\in\mc D_{\nu(d)}\cap \mf S_\mu}\mf S_{\nu(d)}\sigma
\end{equation}
implies $x_\mu=x_{\nu(d)}\sum_{\sigma\in\mc D_{\nu(d)}\cap \mf S_\mu}T_\sigma.$
Hence, by \eqref{TAB} and Corollary \ref{ylaTd}, one can deduce that
\begin{align}
T_{A|B}=&\sum_{u'\in\mc D^{-1}_{\nu(d^{-1})}\cap\mf S_\la}T_{u'}T_{d}x_{\nu(d)}c_{A|B}\sum_{\sigma\in\mc D_{\nu(d)}\cap \mf S_\mu}T_\sigma\notag\\
=&\sum_{u'\in\mc D^{-1}_{\nu(d^{-1})}\cap\mf S_\la}T_{u'}T_{d}c'_{A|B}x_{\nu(d)}\sum_{\sigma\in\mc D_{\nu(d)}\cap \mf S_\mu}T_\sigma
=\sum_{u'\in\mc D^{-1}_{\nu(d^{-1})}\cap\mf S_\la}T_{u'}T_d c'_{(A|B)} x_\mu,\label{TAB-expand3}
\end{align}
proving $T_{A|B}\in \HCR x_\mu$.

We now prove the set is linearly independent. Suppose
$$\sum_{(A|B)\in M_n(\N|\Z_2)_{\la,\mu}}f_{A|B}T_{A|B}=0.$$
By the new basis $\mathscr B'$ on Lemma~\ref{HC-basis2} and \eqref{yyy}, we have, for every $d\in\mc D_{\la,\mu}$,
\begin{equation}\label{linearcombTAB}
\sum_{(A|B)\in M_n(\N|\Z_2)^d_{\la,\mu}}f_{A|B}T_{A|B}=0
\end{equation}
Let $w_\la^0$ be the longest element in the Young subgroup $\mf S_\la$.
By \eqref{decomp-Smu} and the fact that $\ell(w\sigma)=\ell(w)+\ell(\sigma)$ for $w\in\mf S_{\nu(d)}$ and $\sigma\in\mc D_{\nu(d)}\cap\mf S_\mu$,
there exists a unique longest element $\sigma_0$ in $\mc D_{\nu(d)}\cap\mf S_\mu$ such that $w_\mu^0=w^0_{\nu(d)}\sigma_0$.
Suppose $(A|B)\in M_n(\N|\Z_2)^d_{\la,\mu}$.
Then by~\eqref{xla}, Lemma~\ref{TuTdTw} and \eqref{HC-PBW}, we obtain
\begin{equation}\label{TAB-expand2}
\aligned
T_{A|B}&=T_{w^0_\la}T_{d}c_{A|B}T_{\sigma_0}+\sum_{u\in\fS_\la,\sigma\in\mc D_{\nu_{A|B}}\cap \mf S_\mu\atop u< w^0_\la\text{ or } \sigma<\sigma_0}T_u T_{d}c_{A|B}T_{\sigma}\\
&=
T_{w^0_\la}T_{d}c_{A|B}T_{\sigma_0}+\mf{l}_{A|B},
\endaligned
\end{equation}
where 
$\mf l_{A|B}$ is an $R$-linear combination of the elements $T_{ud}c_{A|B}T_\sigma$ with
$u\in\mf S_\la,\sigma\in\mc D_{\nu(d)}\cap \fS_\mu$ and $\ell(ud\sigma)<\ell(w^0_\la d\sigma_0)=\ell(w^0_\la)+\ell(d)+\ell(\sigma_0)$.  This together with~\eqref{linearcombTAB} gives rise to
\begin{align}\label{sum-ABd}
\sum_{(A|B)\in M_n(\N|\Z_2)^d_{\la,\mu}}f_{A|B}T_{w^0_\la d}c_{A|B}T_{\sigma_0}
+\sum_{(A|B)\in M_n(\N|\Z_2)^d_{\la,\mu}}f_{A|B}\mf l_{A|B}=0.
\end{align}
By Lemma~\ref{HC-basis2}, one can deduce that
\begin{equation}\label{hABcAB}
\sum_{(A|B)\in M_n(\N|\Z_2)^d_{\la,\mu}}f_{A|B}T_{w^0_\la d}c_{A|B}T_{\sigma_0}=0.
\end{equation}

Recall the bijection $\j$ in~\eqref{mapj}.
Let $M^d=(m^d_{ij})\in M_n(\N)$ be the unique matrix such that $\j(M^d)=(\la,d,\mu)$.
Then $A+B=M^d$ for all $(A|B)\in M_n(\N|\Z_2)^d_{\la,\mu}$ and we have a bijection between the sets
$M_n(\N|\Z_2)^d_{\la,\mu}$ and $\{B\in M_n(\Z_2)\mid \al_B\leq \nu(d)\}$
where $\nu(d)=(m^d_{11},\ldots,m^d_{n1},\ldots,m^d_{1n},\ldots,m^d_{nn})$ and $\alpha_B$ as defined in \eqref{cAB}.
Now, it is easy to see that the set $\{c_{A|B}=c^{\al_B}_{\nu(d)}\mid B\in M_n(\Z_2),\al_B\leq \nu(d)\}$ is linearly independent in $\mc C_r$.
Therefore, by \eqref{hABcAB} and Lemma~\ref{HC-basis2} again, we have $f_{A|B}=0$ for all $ (A|B)\in M_n(\N|\Z_2)^d_{\la,\mu}$.
\end{proof}
\begin{prop}\label{hom-basis}
Suppose $\la,\mu\in\La(n,r)$. Then the intersection $x_\la\HCR\cap\HCR x_\mu$ is a free $R$-module with basis
$$
\{T_{A|B}\mid (A|B)\in M_n(\N|\Z_2)_{\la,\mu}\}.
$$
\end{prop}
\begin{proof}
By Lemma~\ref{linear-ind}, it suffices to show that the set spans $x_\la\HCR\cap\HCR x_\mu$.
Take an arbitrary $z\in x_\la\HCR\cap\HCR x_\mu$. By Corollary~\ref{trivial-intersect},  we may write
$$
z=\sum_{(d,\sigma,\al)\in X_{\la,\mu}}a_{d,\sigma,\al}\,x_\la T_dc^\al T_\sigma\quad\text{for some}\quad
a_{d,\sigma,\al}\in R.
$$
This together with Corollary \ref{ylaTd} gives rise to
$$
z=\sum_{d}\bigg(\sum_{u'\in\mc D^{-1}_{\nu(d^{-1})}\cap \mf S_\la} T_{u'}\bigg)T_d\sum_{\sigma,\al}a_{d,\sigma,\al}\,x_{\nu(d)}\,c^\al T_\sigma.
$$
For each $d\in\mc D_{\la,\mu}$, write
\begin{equation}\label{zd}
z_d=\sum_{\sigma\in\mc D_{\nu(d)}\cap\mf S_\mu,\al\in\Z_2^r}a_{d,\sigma,\al}\,x_{\nu(d)}\,c^\al T_\sigma\in\sH_{\mu,R}^{\mathsf c}.
\end{equation}
Then
\begin{equation}\label{z-zd}
z=\sum_{d\in\mc D_{\la,\mu}}\sum_{u'\in\mc D^{-1}_{\nu(d^{-1})}\cap \mf S_\la}T_{u'}T_d \,z_d.
\end{equation}
Since $z\in\HCR x_\mu$, by Lemma~\ref{ylaHC}, we have $zT_{s_k}=qz$ for $s_k\in\mf S_\mu$ and hence,
\begin{equation}\label{sum-zero}
\sum_{d\in\mc D_{\la,\mu}}\sum_{u'\in\mc D^{-1}_{\nu(d^{-1})}\cap \mf S_\la}T_{u'}T_d (z_dT_{s_k}-qz_d)=0
\end{equation}
Since $z_d,z_dT_{s_k}\in \mc H^c_{\mu,R}$ by \eqref{zd} for $s_k\in\mf S_\mu$, one can deduce that
\begin{equation}\label{z-sum-subterm}
T_{u'}T_d (z_dT_{s_k}-qz_d)\in T_{u'}T_d\mc H^c_{\mu,R}.
\end{equation}
By the second direct sum decomposition in Corollary~\ref{trivial-intersect},  we obtain
$
T_{u'}T_d (z_dT_{s_k}-qz_d)=0.
$
Writing $z_dT_{s_k}-qz_d$ as a linear combination of $T_wc^\alpha$ with $w\in \fS_\mu,\alpha\in\Z_2^r$, linear independence of
$T_{u'}T_dT_wc^\alpha$ (Lemma~\ref{stdbs}) implies that $z_dT_{s_k}-qz_d=0$,
for $s_k\in\mf S_\mu$ and $d\in\mc D_{\la,\mu}$.
By Lemma~\ref{ylaHC} again, we have
$z_d\in\HCR x_\mu \cap \mc H^{\mathsf c}_{\mu,R}=\mc H^{\mathsf c}_{\mu,R}x_\mu.$
Thus, we can write $z_d$ for each $d\in\mc D_{\la,\mu}$ as
$$
z_d=\sum_{\beta\in\Z_2^r}f^d_\beta \, c^\beta\, x_\mu=\sum_{\beta\in\Z_2^r}f^d_\beta\, c^\beta
x_{\nu(d)}\sum_{\sigma\in\mc D_{\nu(d)}\cap\mf S_\mu}T_\sigma
$$
for some $f^d_{\beta}\in R$ by \eqref{decomp-Smu}.
Hence, by \eqref{zd}, we obtain
$$
\sum_{\sigma\in\mc D_{\nu(d)}\cap\mf S_\mu,\al\in\Z_2^r}a_{d,\sigma,\al}\,x_{\nu(d)}\,c^\al T_\sigma=z_d=\sum_{\beta\in\Z_2^r}f^d_\beta \,c^\beta \,x_{\nu(d)}\sum_{\sigma\in\mc D_{\nu(d)}\cap\mf S_\mu}T_\sigma.
$$
This means
$$
\sum_{\sigma\in\mc D_{\nu(d)}\cap\mf S_\mu}\big(\sum_{\al\in\Z_2^r}a_{d,\sigma,\al}\,x_{\nu(d)}\,c^\al- \sum_{\beta\in\Z_2^r}f^d_\beta\, c^\beta\, x_{\nu(d)}\big)T_{\sigma}=0.
$$
But the left hand side belongs to $\sH_{\mu,R}^{\mathsf c}=\bigoplus_{\sigma\in \mc D_{\nu(d)}\cap\fS_\mu}\sH_{\nu(d),R}^{\mathsf c}T_\sigma,$ forcing every summand is 0. Consequently,
\begin{equation}\label{y-nud}
\sum_{\al\in\Z_2^r}a_{d,\sigma,\al}\,x_{\nu(d)}\,c^\al- \sum_{\beta\in\Z_2^r}f^d_\beta\, c^\beta \,x_{\nu(d)}=0
\end{equation}
for each $\sigma\in\mc D_{\nu(d)}\cap\mf S_\mu$.
Therefore,
\begin{equation}\label{sum-intersect}
\sum_{\al\in\Z_2^r}a_{d,\sigma,\al}\,x_{\nu(d)}\,c^\al=\sum_{\beta\in\Z_2^r}f^d_\beta \,c^\beta \,x_{\nu(d)}
\in x_{\nu(d)}\,\mc H^c_{\nu(d),R}\cap \mc H^c_{\nu(d),R}\,x_{\nu(d)}
\end{equation}
for each $d\in\mc D_{\la,\mu}$ and $\sigma\in\mc D_{\nu(d)}\cap\mf S_\mu$.

On the other hand, since $d\in\mc D_{\la,\mu}$, there exists a unique $M^d=(m^d_{ij})\in M_n(\N)$ such that $\j(M^d)=(\la,d,\mu)$
with $\nu(d)=(m^d_{11},\ldots,m^d_{n1},\ldots,m^d_{1n},\ldots,m^d_{nn})$. (Recall $\nu(d)$ is defined by $\fS_{\nu(d)}=d^{-1}\fS_\la d\cap\fS_\mu$.)
Then, by  \eqref{sum-intersect} and Corollary \ref{tensor-on-part-case},
\begin{equation}\label{sum-intersect1}
\sum_{\al\in\Z_2^r}a_{d,\sigma,\al}\,x_{\nu(d)}\,c^\al=\sum_{B\in M_n(\Z_2), B\leq M^d}g^d_{B}\,x_{\nu(d)}\,c^{\al_B}_{\nu(d)}
\end{equation}
for some $g^d_{B}\in R$, where $B=(b_{ij})\in M_n(\Z_2)\leq M^d$ means $b_{ij}\leq m^d_{ij}$ for $1\leq i,j\leq n$ and $\al_B=(b_{11},\ldots,b_{n1},\ldots,b_{1n},\ldots,b_{nn})\in\Z_2^{n^2}$.

For each $B\in M_n(\Z_2)$ with $B\leq M^d$, we let $A(d,B)=M^d-B\in M_n(\Z_2)$ and then $d=d_{A(d,B)+B}$ and $(A(d,B)|B)\in M_n(\N|\Z_2)^d_{\la,\mu}$.
Then by \eqref{cAB} and \eqref{sum-intersect1} we obtain
$$
\sum_{\al\in\Z_2^r}a_{d,\sigma,\al}\,x_{\nu(d)}\,c^\al=
\sum_{B\in M_n(\Z_2), B\leq M^d}g^d_{B}\,x_{\nu(d)}\,c_{A(d,B)|B}
$$
and hence, by \eqref{zd},
$$
z_d=\sum_{\sigma\in\mc D_{\nu(d)}\cap\mf S_\mu}
\sum_{B\in M_n(\Z_2), B\leq M^d}g^d_{B}\,x_{\nu(d)}\,c_{A(d,B)|B}T_\sigma.
$$
Therefore, by \eqref{z-zd}, one can deduce that
\begin{align*}
z=&\sum_{d\in\mc D_{\la,\mu}}\sum_{u'\in\mc D^{-1}_{\nu(d^{-1})}\cap \mf S_\la}T_{u'}T_d\sum_{\sigma\in\mc D_{\nu(d)}\cap\mf S_\mu}
\sum_{B\in M_n(\Z_2), B\leq M^d}g^d_{B}\,x_{\nu(d)}\,c_{A(d,B)|B}\,T_\sigma\\
=&\sum_{d\in\mc D_{\la,\mu}}\sum_{B\in M_n(\Z_2), B\leq M^d}g^d_{B}\sum_{u'\in\mc D^{-1}_{\nu(d^{-1})}\cap \mf S_\la}T_{u'}T_d\,x_{\nu(d)}\,c_{A(d,B)|B}\sum_{\sigma\in\mc D_{\nu(d)}\cap\mf S_\mu}T_\sigma\\
=&\sum_{d\in\mc D_{\la,\mu}} \sum_{B\in M_n(\Z_2), B\leq M^d}g^d_{B}\,x_\la\, T_d\, c_{A(d,B)|B}\,\sum_{\sigma\in\mc D_{\nu(d)}\cap\mf S_\mu}T_\sigma\\
=&\sum_{d\in\mc D_{\la,\mu}}\sum_{B\in M_n(\Z_2), B\leq M^d}g^d_{B}\,T_{A(d,B)|B}
\end{align*}
since $d=d_{A(d,B)+B}$. This proves the proposition.
\end{proof}
For $(A|B)\in M_n(\N|\Z_2)_r$, define
$\phi_{(A|B)}\in \mc Q_q(n,r;R)={\rm End}_{\HCR}(\oplus_{\mu\in\La(n,r)}x_\mu\HCR)$ via
\begin{equation}\label{tilde-phi}
\phi_{(A|B)}(x_\mu h)=\delta_{\mu,{\rm co}(A+B)} T_{A|B}h
\end{equation}
for $\mu\in\La(n,r)$ and $h\in\HCR$.
\begin{thm}\label{QqSchur-basis}
Let $R$ be a commutative ring of characteristic not equal to 2.
Then the algebra $\mc Q_q(n,r;R)$ is a free $R$-module with a basis given by the set
$$
\{\phi_{(A|B)}\mid (A|B)\in M_n(\N|\Z_2)_r\}.
$$
\end{thm}
\begin{proof}
By \eqref{QqSchur-defn}, one can deduce that
$
\mc Q_q(n,r;R)=\bigoplus_{\la,\mu\in\La(n,r)}{\rm Hom}_{\HCR}(x_\mu\HCR,x_\la\HCR).
$
Then, by Lemma~\ref{hom-interc} and Proposition~\ref{hom-basis}, the set
$$
\{\phi_{(A|B)}\mid {\rm ro}(A+B)=\la,{\rm co}(A+B)=\mu\}
$$
is an $R$-basis for ${\rm Hom}_{\HCR}(x_\mu\HCR,x_\la\HCR)$ for each pair $\la,\mu\in\La(n,r)$.
Therefore, the set
$$
\{\phi_{(A|B)}\mid (A|B)\in M_n(\N|\Z_2)_r\}
$$
forms an $R$-basis for $\mc Q_q(n,r;R)$.
\end{proof}
By Theorem~\ref{QqSchur-basis}, we have the following base change property for $\mc Q_q(n,r;R)$.

\begin{cor}  Maintain the assumption on $R$ as above. Suppose that $R$  is an $\sA$-algebra via $\bsq\mapsto q$.
Then
$$
\mc Q_q(n,r;R)\cong \mc Q_\bsq(n,r)_R:=\mc Q_\bsq(n,r)\otimes_\sA R.
$$
\end{cor}

\section{Identification with the quotients of the quantum queer supergroup}

In this section, we shall show that the \qq-Schur superalgebra $\mc Q_\bsv(n,r)$ coincides with the quantum queer Schur superalgebra constructed in \cite{DW}.
In particular, they are homomorphic images of the quantum queer supergroup $U_\bsv(\mathfrak q_n)$.

Let $q=v^2$ for some $v\in R$.
Let $V(n|n)_R=V_0\oplus V_1$ be a free $R$-supermodule with basis $e_1,\ldots,e_n$ for $V_0$ and basis $e_{-1},\ldots,e_{-n}$ for $V_1$. With the fixed ordered basis $\{e_1,\ldots,e_n,e_{-1},\ldots,e_{-n}\}$, we often identify the $R$-algebra $\End_{R}(V(n|n)_R)$ of all $R$-linear maps on $V(n|n)_R$ with the $2n\times2n$ matrix algebra $M_{2n}(R)$ over $R$ and, hence, $\End_{R}(V(n|n)_R^{\otimes r})$ with $M_{2n}(R)^{\otimes r}$.

Let
$$\aligned
I(n|n)&=\{1,2,\ldots,n,-1,-2,\ldots,-n\},\\
\Innr&=\{\udi=(i_1,\ldots,i_r)\mid i_k\in I(n|n), 1\leq k\leq r\}.
\endaligned
$$
Then, $\fS_r$ acts on $\Innr$ by the place permutation
$$\udi s_k=(i_1,\ldots,i_{k-1},i_{k+1},i_k,i_{k+2},\ldots,i_r). $$
For $\udi=(i_1,\ldots,i_r)\in \Innr$, set
$
e_{\udi}=e_{i_1}\otimes e_{i_2}\otimes\cdots\otimes e_{i_r},
$
then the set $\{e_{\udi}\mid \udi\in \Innr\}$ forms a basis for $V(n|n)_R^{\otimes r}$.

Denote by $E_{i,j}$ for $i,j\in I(n|n)$ the standard elementary matrix with the $(i,j)$th entry being 1 and 0 elsewhere.
Then $\{E_{i,j}\mid i,j\in I(n|n)\}$ can be viewed as the standard basis for $\End_{R}(V(n|n)_R)$, that is, $E_{i,j}(e_k)=\delta_{j,k}e_i$. Following \cite{Ol}, we set
\begin{equation}\label{operator-OmegaTS}
\aligned
\Omega&=  \sum_{1\leq a\leq n}(E_{-a,a}-E_{a,-a}),\\
T&=  \sum_{i,j\in I(n|n)}(-1)^{\widehat{j}}E_{i,j}\otimes E_{j,i},\\
S&= v\sum_{i\leq j\in I(n|n)}S_{i,j}\otimes E_{i,j}\in\End_{R}
(V(n|n)_R^{\otimes 2}),
\endaligned
\end{equation}
where $\widehat{j}=\widehat{e}_j=0$ if $j>0$ and $\widehat{j}=\widehat{e}_j=1$ if $j<0$, and $S_{i,j}$ for $i\leq j$ are defined as follows:
\begin{equation}\label{S}
\aligned
S_{a,a}&=1+(v-1)(E_{a,a}+E_{-a,-a}),\quad 1\leq a\leq n,\\
S_{-a,-a}&=1+(v^{-1}-1)(E_{a,a}+E_{-a,-a}),\quad1\leq a\leq n, \\
S_{b,a}&=(v-v^{-1})(E_{a,b}+E_{-a,-b}),\quad 1\leq b<a\leq n, \\
S_{-b,-a}&=-(v -v^{-1})(E_{a,b}+E_{-a,-b}),\quad 1\leq a<b\leq n, \\
S_{-b,a}&=-(v -v^{-1}) (E_{-a,b}+E_{a,-b}),\quad 1\leq a, b\leq n.
\endaligned
\end{equation}

To endomorphisms $A\in\End_{R}(V(n|n)_R)$ and
$Z=\sum_{t}X_{t}\otimes Y_{t}\in\End_{R}(V(n|n)_R)^{\otimes
2}=\End_{R}(V(n|n)_R^{\otimes
2})$, we associate the following elements in $\End_{R}(V(n|n)_R^{\otimes
r})$:
\begin{align*}
A^{(k)}&={\rm id}^{\otimes (k-1)}\otimes A\otimes {\rm id}^{\otimes(r-k)},\qquad 1\leq k\leq r,\\
Z^{(j,k)}&=\sum_{t}X_{t}^{(j)}Y_{t}^{(k)}, \qquad 1\leq j\neq k\leq r.
\end{align*}
Let $\bar{S}=TS$. It follows from \cite[Theorems 5.2-5.3]{Ol} that there exists a {\it left} $\HCR$-supermodule structure on $V(n|n)^{\otimes r}_R$ given by
\begin{equation}\label{left-act}
\aligned
\Psi_r: \HCR&\longrightarrow{\rm End}_R(V(n|n)^{\otimes r}_R), \quad T_i\longmapsto\bar{S}^{(i,i+1)}, c_j\longmapsto\Omega^{(j)}
\endaligned
\end{equation}
for all $1\leq i\leq r-1,1\leq j\leq r$.
Here we remark that the even generators $T_1,\ldots,T_{r-1}$ for $\HCR$
are related to the even generators $t_1,\ldots,t_{r-1}$ in \cite{Ol} via $T_i=vt_i$ for $1\leq i\leq r-1$.
This above action has been explicitly worked out in \cite[Lemma 3.1]{WW} via \eqref{operator-OmegaTS}:

\begin{lem}\label{WW} The $\HCR$-supermodule structure on $V(n|n)^{\otimes r}_R$ is given by the following formulas: for $\udi=(i_1,\ldots,i_r)\in \Innr$ and $1\leq k\leq r-1$, $1\leq j\leq r$,
$$\aligned
c_j\cdot e_{\udi}&=(-1)^{\hat{e}_{i_1}+\cdots+\hat{e}_{i_{j}}}e_{i_1}\otimes\cdots\otimes e_{i_{j-1}}\otimes
e_{-i_j}\otimes e_{i_{j+1}}\otimes\cdots\otimes e_{i_r}.\\
T_k\cdot e_{\udi}&=\left\{
\begin{array}{ll}
v^2e_{\udi s_k}+(v^2-1)e_{\udi^-_k},&\text{ if }i_k=i_{k+1}\geq 1,\\
-e_{\udi s_k},&\text{ if }i_k=i_{k+1}\leq -1,\\
e_{\udi s_k},&\text{ if }i_k=-i_{k+1}\geq 1,\\
v^2e_{\udi s_k}+(v^2-1)e_{\udi},&\text{ if }i_k=-i_{k+1}\leq -1,\\
ve_{\udi s_k}+(v^2-1)e_{\udi^-_k}+(v^2-1)e_{\udi},&\text{ if }|i_k|<|i_{k+1}| \text{ and }i_{k+1}\geq 1,\\
(-1)^{\widehat{i}_k}ve_{\udi s_k},&\text{ if }|i_k|<|i_{k+1}| \text{ and }i_{k+1}\leq -1,\\
ve_{\udi s_k}+(-1)^{\widehat{i}_{k+1}}(v^2-1)e_{\udi^-_k},&\text{ if }|i_{k+1}|<|i_{k}| \text{ and }i_{k}\geq 1,\\
(-1)^{\widehat{i}_{k+1}}e_{\udi s_k}+(v^2-1)e_{\udi},&\text{ if }|i_{k+1}|<|i_k| \text{ and }i_{k}\leq-1,
\end{array}
\right.\\
\endaligned
$$
where $\udi^-_k=(i_1,\ldots,i_{k-1},-i_k,-i_{k+1},i_{k+2},\ldots,i_r).$
\end{lem}

For each $\udi\in \Innr$, define ${\rm wt}(\udi)=\la=(\la_1,\ldots,\la_n)\in\La(n,r)$ by setting
$$
\la_a=|\{k\mid a=|i_k|, 1\leq k\leq r\}|, \quad\forall 1\leq a\leq n.
$$
For each $\la\in\La(n,r)$, define $\udi_\la$ by
$$
\udi_\la=(\underbrace{-1,\ldots,-1}_{\la_1},\underbrace{-2,\ldots,-2}_{\la_2},\ldots,\underbrace{-n,\ldots,-n}_{\la_n}).
$$
Recall from \eqref{xla} the elements $y_\la$.
\begin{cor}\label{tensorspace}
The following holds as left $\HCR$-supermodules:
$$
V(n|n)_R^{\otimes r}\cong \bigoplus_{\la\in\La(n,r)}\HCR y_\la.
$$
\end{cor}
\begin{proof} For each $\la\in\La(n,r)$,
let $(V(n|n)_R^{\otimes r})_\la$ be the $R$-submodule of $V(n|n)_R^{\otimes r}$ spanned by the elements $e_{\udi}$ with $\udi\in \Innr$ and  ${\rm wt}(\udi)=\la$.
Then, by Lemma~\ref{WW}, $(V(n|n)_R^{\otimes r})_\la$ is stable under the action of $\HCR$ and, moreover, $V(n|n)_R^{\otimes r}$ can
be decomposed as
\begin{equation}\label{decompVr}
V(n|n)^{\otimes r}_R=\bigoplus_{\la\in\La(n,r)}(V(n|n)_R^{\otimes r})_\la.
\end{equation}
Clearly $e_{\udi_\la}\in (V(n|n)_R^{\otimes r})_\la$.
For each $w\in\mf S_r$, we have $\udi_\la w=(i_{w(1)},\ldots,i_{w(r)})$, where we write $\udi_\la=(i_1,\ldots,i_r)$.
Then we can easily deduce that
$$
\{e_{\udi}\mid \udi\in \Innr,{\rm wt}(\udi)=\la, i_k\leq -1,1\leq k\leq r\}=\{e_{\udi_\la w}\mid w\in\mf S_r\}=
\{e_{\udi_\la d}\mid d\in\mc D_{\la}\}.
$$
Moreover, by Lemma~\ref{WW}, one can deduce that
$
\mc H_{r,R}e_{\udi_\la}=R\text{-span}\{e_{\udi_\la d}\mid d\in\mc D_{\la}\}.
$
Meanwhile by Frobenius reciprocity there exists an $\HCR$-homomorphism
\begin{equation}\label{wt-space}
\HCR \otimes_{\mc H_{r,R}}\mc H_{r,R}e_{\udi_\la}\rightarrow \HCR e_{\udi_\la}=(V(n|n)_R^{\otimes r})_\la,\quad c^{\al}\otimes e_{\udi_\la d}\mapsto c^\al \cdot e_{\udi_\la d}
\end{equation}
which is bijective due to the fact that $\{c^\al \cdot e_{\udi_\la d}\mid d\in\mc D_{\la},\al\in\Z_2^r\}$ is a basis for $(V(n|n)_R^{\otimes r})_\la$.
Observe that the tensor product $V(0|n)_R^{\otimes r}$ of odd $R$-submodule $V(0|n)_R$ of $V(n|n)_R$ admits a basis $\{e_{\udi}\mid \udi\in \Innr, i_k\leq -1,1\leq k\leq r\}$, which can be identified with $\bigoplus_{\la\in\La(n,r)}\mc H_{r,R} e_{\udi_\la}$.
Then, by Lemma~\ref{WW} and \cite[Lemma~2.1]{Du}, there is an $\sH_{r,R}$-module isomorphism
\begin{equation}\label{yla}
\mc H_{r,R} e_{\udi_\la}\cong \mc H_{r,R} y_\la.
\end{equation}
Taking a direct sum over $\La(n,r)$ and induction to $\HCR$ and noting \eqref{decompVr} and \eqref{wt-space} gives the required isomorphism.
\end{proof}

Recall the anti-involution $\tau$ defined in~Lemma \ref{anti-inv}(3) and the twist functor in Remark~\ref{tau twist}.
In particular, $(V(n|n)^{\otimes r}_R)^\tau$ affords a right $\HCR$-supermodule. By \eqref{iota-twist}, we immediately have  the following.
\begin{cor}\label{tensorspace2}
There is an isomorphism of right $\HCR$-supermodules:
$$
(V(n|n)_R^{\otimes r})^\tau\cong \bigoplus_{\la\in\La(n,r)}x_\la\HCR,
$$
which induces an algebra isomorphism
\begin{equation*}
{\rm End}_{\HCR}(V(n|n)_R^{\otimes r})\cong {\rm End}_{\HCR}\big((V(n|n)_R^{\otimes r})^\tau\big)\cong{\rm End}_{\HCR}\bigg(\bigoplus_{\la\in\La(n,r)}x_\la\HCR\bigg)= \mc Q_{v^2}(n,r;R).
\end{equation*}
\end{cor}
Hence, if $R=\sZ=\mathbb Z[\bsv,\bsv^{-1}]$, the endomorphism superalgebra ${\rm End}_{\HCR}(V(n|n)_R^{\otimes r})$ is the quantum queer Schur superalgebra $\mc Q_\bsv(n,r)$ considered in \cite{DW}.\footnote{In \cite{DW}, we used the indeterminate $q$ instead of $\bsv$ and the notation $\mc Q_q(n,r)$ denotes the queer $q$-Schur superalgebra over the rational function field $\mathbb C(q)$.}  The isomorphism above shows that our notation is consistent as, by \eqref{v-Schur},  we used the same notation to denote the right hand side $\mc Q_{\bsv^2}(n,r;\sZ)$;
see \eqref{v-Schur}.

Olshanski \cite{Ol} introduced the quantum deformation $U_\bsv(\mf q_n)$ over $\C(\bsv)$ of the universal enveloping algebra of
the queer Lie superalgebra $\mf q(n)$ and defined a superalgebra homomorphism
$$\Phi_r: U_\bsv(\mf q_n)\longrightarrow {\rm End}_{\C(\bsv)}(V(n|n)^{\otimes r}_{\C(\bsv)}).$$
We refer the reader to \cite{Ol} for the details of the definitions of $U_\bsv(\mf q(n))$ and $\Phi_r$.
The following result is known as the double centraliser property and forms the first part of the quantum analog of Schur-Weyl--Sergeev duality for $U_\bsv(\mf q_n)$ and $\mc H^\sfc_{r,\C(\bsv)}$.
Recall also the algebra homomorphism $\Psi_r$ defined in \eqref{left-act}.

\begin{prop}[{\cite[Theorem 5.3]{Ol}\label{Ol}}]\label{Olthm}
The algebras $\Phi_r(U_\bsv(\mf q_n))$ and $\Psi_r(\mc H^\sfc_{r,\C(\bsv)})$ form mutual centralisers in ${\rm End}_{\C(\bsv)}(V(n|n)^{\otimes r}_{\C(\bsv)})$. In particular, $\Phi_r$ induces an algebra epimorphism
$$\Phi_r:U_\bsv(\mf q_n)\longrightarrow {\rm End}_{\sH^\sfc_{r,\C(\bsv)}}(V(n|n)^{\otimes r}_{\C(\bsv)})=\vQnr_{\C(\bsv)}.$$
\end{prop}

Thus,
by \cite[Theorem 9.2]{DW}, the \qq-Schur superalgebra $\vQnr_{\C(\bsv)}$
has a presentation with even generators $K_i^{\pm1},E_j,F_j$ and odd generators $K_{\bar i},E_{\bar j},F_{\bar j}$, for $1\leq i\leq n$ and $1\leq j\leq n-1$, subject to the relations (QQ1)--(QQ9) given in \cite{DW}.

\begin{rem} \label{ForMori2}
It would be natural to expect that one may use the integral basis $\{\phi_{(A|B)}\}_{(A|B)}$ given in Theorem \ref{QqSchur-basis} to describe the images of the generators $K_i^{\pm1},E_j,F_j,K_{\bar i},E_{\bar j},F_{\bar j}$ and thus, to obtain a realisation of  $U_\bsv(\mf q_n)$ similar to the ones given in \cite{BLM,DG}.
\end{rem}

\section{Irreducible $\vQr_\K$-supermodules}
In this section, we shall give a construction of all irreducible $\vQr_\K$-supermodules over certain field extension $\K$ of $\mcZ=\mathbb Z[\bsv,\bsv^{-1}]$. But we first look at a few general facts for the superalgebra $\qQnr$ over $\sA=\mathbb Z[\bsq]$.

Let $\omega=(1^r)$. In the case $n\geq r$,  we can view $\omega$  as an element in $\La(n,r)$ and define
$$
e_\omega=\phi_{(A_\omega|B_\omega)}\in \qQnr,
$$
where $A_\omega={\rm diag}(\underbrace{1,\ldots,1}_{r},0,\ldots,0)$ and $B_\omega=0$.
Then, by \eqref{TAB} and \eqref{tilde-phi},
\begin{equation}\label{eomega}
e_\omega(x_\la h)=\delta_{\omega,\la}x_\la h,\quad e^2_\omega=e_\omega
\end{equation}
for $\la\in\La(n,r),h\in\mc H^c_r$.

Recall from \cite[Theorem 2.24]{DJ2} that the $q$-Schur algebra $\qSnr$ is defined as
\begin{equation}\label{qSchur}
\qSnr:={\rm End}_{\mc H_r}\bigg(\bigoplus_{\la\in\La(n,r)}x_\la\mc H_r\bigg).
\end{equation}
Let
$$\sT(n,r)=\bigoplus_{\la\in\La(n,r)}x_\la\mc H_r\quad\text{and}\quad
\sT^\sfc(n,r)=\bigoplus_{\la\in\La(n,r)}x_\la\mc H_r^\sfc.$$
Since $\sT^\sfc(n,r)=\sT(n,r)\otimes_{\mc H_r}\HC$ (see \eqref{ind}), there is a natural algebra embedding
\begin{align*}
\epsilon: \qSnr\longrightarrow \qQnr, f\longmapsto f\otimes 1.
\end{align*}
Equivalently, the $q$-Schur algebra $\qSnr$ can be identified as a subalgebra of $\qQnr$ via the following way
\begin{equation}\label{embedding}
\qSnr=\{\phi\in \qQnr\mid \phi(x_\la)\in\mc H_r, \forall \la\in\La(n,r)\}.
\end{equation}
On the other hand, if $n\geq r$, the evaluation map gives algebra isomorphisms
$$e_\omega \qQnr e_\omega\cong \mc H^\sfc_{r}\quad\text{and}\quad
e_\omega \qSnr e_\omega\cong \mc H_{r}.$$
We will identify them in the sequel.  In particular, $\qQnr e_\omega$ is a $\qQnr$-$\HC$-bisupermodule.

\begin{lem}\label{QqSchur-eomega}
If $n\geq r$, then there is a $\qQnr$-$\HC$-bisupermodule isomorphism
$$\qQnr e_\omega\cong\sT^\sfc(n,r).$$
Moreover, restriction gives an $\qSnr$-$\HC$-bisupermodule isomorphism
$$
\qQnr e_\omega\cong \qSnr e_\omega\otimes_{\mc H_r}\HC.
$$
\end{lem}
\begin{proof}
By \eqref{eomega} and \eqref{embedding},  $e_\omega\in \qSnr$ and  $x_\omega=1$. Hence, the evaluation map gives an $\qSnr$-$\sH_r$-bimodule isomorphism
$$\qSnr e_\omega={\rm Hom}_{\mc H_r}(x_\omega\mc H_r,\sT(n,r))\cong \sT(n,r),$$
and a $\qQnr$-$\HC$-bisupermodule isomorphism
$$\qQnr e_\omega={\rm Hom}_{\HC}(x_\omega\HC,\sT^\sfc(n,r))\cong \sT^\sfc(n,r).$$
Thus, the restriction gives an $\qSnr$-$\HC$-bisupermodule isomorphism
\begin{align*}
\qQnr e_\omega\cong \sT(n,r)\otimes_{\mc H_r}\HC\cong \qSnr e_\omega\otimes_{\mc H_r}\HC.
\end{align*}
This proves the lemma.
\end{proof}

Assume that $\F$ is a field which is a $\mcZ$-algebra such that the image of $\bsq$
is not a root of unity. Thus, both $\mc H_{r,\F}$ and $\qSnr_{\F}=\qSnr\otimes_\mcZ\F$ are semisimple $\F$-algebras.
The following result generalizes the category equivalence between $\qSnr_\F$-mod and $\mc H_{r,\F}$-mod in the case $n\geq r$
and the proof is somewhat standard, see, e.g., \cite[Theorem 4.1.3]{DDF}.

\begin{prop}\label{morita-equiv}
Assume $n\geq r$. The superalgebras $\qQnr_\F$ and $\mc H^\sfc_{r,\F}$ are Morita equivalent.
\end{prop}
\begin{proof}We consider the two functors generally defined in \cite[\S6.2]{Gr}.
Firstly, the $\qQnr_\F$-$\mc H^\sfc_{r,\F}$-bisupermodule $\sT^\sfc(n,r)_\F$ induces a functor
\begin{equation}\label{functor F}
\mc F: \mc H^\sfc_{r,\F}\text{-\bf smod}\longrightarrow \qQnr_\F\text{-\bf smod}, \quad L\longmapsto \sT^\sfc(n,r)_\F\otimes_{\mc H^\sfc_{r,\F}} L.
\end{equation}
Secondly,
there exists the functor
\begin{align*}
\mc G: \qQnr_\F\text{-\bf smod}\rightarrow\mc H^\sfc_{r,\F}\text{-\bf smod},\quad M\mapsto e_\omega M.
\end{align*}
It is clear (cf. \cite[(6.2d)]{Gr}) that $\mc G\circ\mc F\cong {\rm id}_{\mc H^\sfc_{r,\F}\text{-\bf smod}}$.
So it suffices to prove $\mc F\circ\mc G\cong {\rm id}_{\qQnr_\F\text{-\bf smod}}$.

It is known from the case for the $q$-Schur algebra that, for any $\qSnr_\F$-module $M$, there is a left $\qSnr_\F$-module isomorphism
$$
\psi: \qSnr_\F e_\omega\otimes_{\mc H_{r,\F}}e_\omega M\cong M
$$
defined by $\psi(x\otimes m)=xm$ for $x\in \qSnr_\F e_\omega$ and $m\in e_\omega M$.
This together with Lemma~\ref{QqSchur-eomega} gives rise to a left $\qSnr_\F$-module isomorphism
$$
\Psi: \qQnr_\F e_\omega\otimes_{\mc H^\sfc_{r,\F}}e_\omega M=\qSnr_\F e_\omega\otimes_{\mc H_{r,\F}}e_\omega M\cong M
$$
We now claim that $\Psi$ is a $\qQnr_\F$-supermodule isomorphism.
Indeed, considering the basis $\{\phi_{(A|B)}\mid (A|B)\in M_n(\N|\Z_2)\}$ for $\qQnr_\F$, we have
$$
\qQnr_\F e_\omega={\rm Hom}_{\mc H^\sfc_{r,\F}}(x_\omega\mc H^\sfc_{r,\F},\bigoplus_{\la\in\La(n,r)}x_\la\mc H^\sfc_{r,\F})
=\bigoplus_{\la\in\La(n,r)}\phi_{(A(\omega,\la)|0)}\mc H^\sfc_{r,\F},
$$
where $A(\omega,\la)=(a^{\omega,\la}_{ij})$ with $a^{\omega,\la}_{ij}=1$ if $\la_1+\cdots+\la_{i-1}+1\leq j\leq \la_1+\cdots+\la_{i-1}+\la_i$ and
$a^{\omega,\la}_{ij}=0$ otherwise.
Clearly $\phi_{(A(\omega,\la)|0)}(x_\omega )=x_\la\in\mc H_{r,\F} $ for $\la\in\La(n,r)$.
Therefore $\phi_{(A(\omega,\la)|0)}\in \qSnr_\F e_\omega$ by \eqref{embedding}.
This implies that
$$
\Psi(\phi_{(A(\omega,\la)|0)}h\otimes m)={\psi}(\phi_{(A(\omega,\la)|0)}\otimes hm)=\phi_{(A(\omega,\la)|0)}hm
$$
for any $h\in\mc H^\sfc_{r,\F}, m\in e_\omega M$. Hence $\Psi(x\otimes m)=xm$ for $x\in \qQnr_\F e_\omega, m\in e_\omega M$.
This means $\Psi$ is a $\qQnr_\F$-supermodule homomorphism.
\end{proof}

We are now ready to look at the classification of irreducible representations of the superalgebra $\vQnr_\K$,
where
$$
\K:=\C(\bsv)\Big(\sqrt{[\![2]\!]},\sqrt{[\![3]\!]},\ldots,\sqrt{[\![r]\!]}\Big)
$$
is the field extension\footnote{The filed $\K$ is denoted by $\F$ in \cite[\S4]{JN}.} of $\C(\bsv)$ with $[\![k]\!]=\frac{\bsq^k-\bsq^{-k}}{\bsq-\bsq^{-1}}$ for $k\in\N$ ($\bsq=\bsv^2$). It is known from \cite[Propoition 2.2]{JN}
 that $\sH_{r,\C(\bsv)}^\sfc$ is a semisimple superalgebra. Thus, $\vQnr_{\C(\bsv)}$ is semisimple.
 The following result will imply that $\K$ is a splitting field of $\vQnr_{\C(\bsv)}$.

A partition $\xi$ of $r$ is said to be {\it strict} if its non-zero parts are distinct.
Denote by $\mc{SP}(r)$ the set of strict partitions $\xi$ of $r$.
For $\xi\in\mc{SP}(r)$, denote by $l(\xi)$ the length of $\xi$ and let
$$
\delta(\xi)=\left\{
\begin{array}{ll}
0,&\text{ if }l(\xi)\text{ is even},\\
1,&\text{ if }l(\xi)\text{ is odd}.
\end{array}
\right.
$$


%
\begin{prop}[{\cite[Corollary 6.8]{JN}}]\label{IrrHC}
For each $\xi\in\mc{SP}(r)$, there exists
a left simple $\mc H^\sfc_{r,\K}$-supermodule $U^\xi$ such that $\{U^\xi\mid\xi\in\mc{SP}(r)\}$
forms a complete set of nonisomorphic irreducible left $\mc H^\sfc_{r,\K}$-supermodules.
Moreover, every $U^\xi$ is split irreducible\footnote{An irreducible supermodule $X$ over an $\F$-superalgebra $\mathscr A$ is called split irreducible if $X\otimes_{\F}\E$ is irreducible over $\mathscr A_\E=\mathscr A\otimes_\F\E$
for any field extension $\E\supseteq\F$.} and $U^\xi$ is of type $\texttt{M}$ (resp., $\texttt{Q}$) if $\delta(\xi)=0$ (resp., $\delta(\xi)=1$).
\end{prop}
If we take $\F=\K$ in \eqref{functor F}, then Propositions \ref{morita-equiv} and \ref{IrrHC} imply the following immediately.
\begin{cor} If $n\geq r$,
then the set $\{\sT^\sfc(n,r)_\K\otimes_{\mc H^\sfc_{r,\K}} U^\xi\mid \xi\in\mc{SP}(r)\} $ is a complete set of nonisomorphic irreducible $\vQnr_\K$-supermodules.
\end{cor}

For positive integers $n,r$, let
$$
\mc{SP}(n,r)=\{\xi\in\mc{SP}(r)\mid \sT^\sfc(n,r)_\K\otimes_{\mc H^\sfc_{r,\K}} U^\xi\neq 0\}.
$$
Clearly, we have from the above $\mc{SP}(n,r)=\mc{SP}(r)$ for all $n\geq r$.

Recall from the last section, there is a commuting action of $U_\bsv(\mf q_n)_\K$ and $\mc H^\sfc_{r,\K}$ on $V(n|n)^{\otimes r}_\K$.
For $\xi\in\mc{SP}(r)$, let
$$V(\xi)=\text{Hom}_{\mc H^c_{r,\K}}(U^{\xi}, V(n|n)^{\otimes r}_\K).$$
This is a $U_\bsv(\mf q_n)_\K$-module.
  If $V(\xi)\neq0$, then $U^\xi$ is a direct summand of $V(n|n)^{\otimes r}_\K$. Hence it is a simple $\vQnr_\K$-supermodule and, hence, a simple $U_\bsv(\mf q_n)_\K$-supermodule.

\begin{lem} Let $\xi\in\mc{SP}(r)$. If $n<r$ and $l(\xi)>n$, then $V(\xi)=0$. In paticular, we have
$$
\mc{SP}(n,r)=\{\xi\in\mc{SP}(r)\mid l(\xi)\leq n\}.
$$
\end{lem}
\begin{proof}By \eqref{decompVr} and \eqref{yla}, we have $V(n|n)^{\otimes r}_\K$ is isomorphic to a direct sum of $\sH^\sfc_{r,\K}y_\la$
with $\la\in\La(n,r)$. By \cite[Lem.~3.50]{CW}, there exist some nonnegative integers $k_{\xi,\la}$ with $k_{\xi,\xi}\neq0$ such that\footnote{Though the result there is for the specialisation $(\HC)_{\mathbb C}$ of $\HC$ at $\bsq=1$.
However, by the category equivalence between the module categories, the proof there carries over.}
\begin{equation}\label{permutation-decomp}
\sH^\sfc_{r,\K}y_\la\cong\oplus_{\xi\in\mc{SP}(r),\xi\trianglerighteq\la}k_{\xi,\la}U^\xi,
\end{equation}
where $\trianglerighteq$ is the dominance order. Since $\xi\trianglerighteq\la$ implies $l(\xi)\leq l(\la)=n$, it follows that $V(\xi)=0$ if $l(\xi)>n$, proving the first assertion.
To prove the last assertion, we write $U^\xi=\sH^\sfc_{r,\K}\epsilon$ for some idempotent $\ep$.
Then, as a left $\vQnr_\K$-module,
$$V(\xi)\cong \ep V(n|n)^{\otimes r}_\K\cong(V(n|n)^{\otimes r}_\K)^\tau\ep\cong\sT^\sfc(n,r)_\K\otimes_{\sH^\sfc_{r,\K}}U^\xi.$$
The last assertion follows from the fact $k_{\xi,\xi}\neq0$.
\end{proof}
We now use Green's \cite[Theorem (6.2g)]{Gr} to get the following classification theorem.
\begin{thm}\label{classification}
For positive integers $n,r$,  the set
$$
\{\sT^\sfc(n,r)_\K\otimes_{\mc H^\sfc_{r,\K}} U^\xi\cong V(\xi)\mid \xi\in\mc{SP}(n,r)\}
$$
forms a complete set of nonisomorphic irreducible $\vQnr_\K$-supermodules.
\end{thm}
\begin{proof}The case for $n\geq r$ is seen above. We now assume $n<r$. Consider the natural embedding
$$
\La(n,r)\hookrightarrow \La(r,r),\quad \la\mapsto \overline{\la}=(\la_1,\ldots,\la_n,0,\ldots,0).
$$
For $\la\in\La(n,r)$, let $D(\la)={\rm diag}(\la_1,\la_2,\ldots,\la_n,0,\ldots,0)\in M_r(\N)$ and set
$$
e=\sum_{\la\in\La(n,r)}\phi_{(D(\la)|0)}.
$$
Then
\begin{align}\label{centralizer-subalgebra}
\vQnr_\K\cong {\rm End}_{\mc H^\sfc_{r,\K}}\bigg(\bigoplus_{\la\in\La(n,r)}x_\la\mc H^\sfc_{r,\K}\bigg)=e {\rm End}_{\mc H^\sfc_{r,\K}}\bigg(\bigoplus_{\mu\in\La(r,r)}x_\mu\mc H^\sfc_{r,\K}\bigg)e=e\mc Q_\bsv(r,r)_\K e.
\end{align}
Meanwhile, we have
\begin{align*}
e(\sT^\sfc(r,r)_\K\otimes_{\mc H^\sfc_{r,\K}}U^\xi)=e\sT^\sfc(r,r)_\K\otimes_{\mc H^\sfc_{r,\K}}U^\xi
=e\bigg(\bigoplus_{\mu\in\La(r,r)}x_\mu\mc H^\sfc_{r,\K}\bigg)\otimes_{\mc H^\sfc_{r,\K}}U^\xi\\
=\bigg(\bigoplus_{\la\in\La(n,r)}x_{\overline{\la}}\mc H^\sfc_{r,\K}\bigg)\otimes_{\mc H^\sfc_{r,\K}}U^\xi
=\sT^\sfc(n,r)_\K\otimes_{\mc H^\sfc_{r,\K}}U^\xi.
\end{align*}
Since $\{\sT^\sfc(r,r)_\K\otimes_{\mc H^\sfc_{r,\K}}U^\xi\mid \xi\in\mc{SP}_r\}$ is a complete set of nonisomorphic simple $\sQ_\bsv(r,r)_\K$-supermodules,
\eqref{centralizer-subalgebra} together with \cite[(6.2g)]{Gr} implies that the set
$$\{\sT^\sfc(n,r)_\K\otimes_{\mc H^\sfc_{r,\K}}U^\xi)\mid \xi\in\mc{SP}_r\}\setminus\{0\}
=\{\sT^\sfc(n,r)_\K\otimes_{\mc H^\sfc_{r,\K}} U^\xi\mid \xi\in\mc{SP}(n,r)\}$$
is a complete set of nonisomorphic simple $\vQnr_\K$-supermodules.
\end{proof}

\begin{rem}\label{ForMori3}(1) By regarding the category $\oplus_{r\geq 0}\vQnr_\K{-\bf smod}$ as a full subcategory of the category of finite dimensional $U_\bsv(\mathfrak q_n)$-supermodules, we recover the category $\mathcal O_{\text{int}}^{\geq0}$ of tensor modules
(the counterpart of polynomial representations studied by Green \cite{Gr}) investigated in \cite{GJKKK}.

(2) We remark that, for the $q$-Schur superalgebras $\mathcal S_\bsv(m|n,r)_{\mathbb F}$ of type $\tt M$ with $m+n\geq r$, their irreducible representations at (odd) roots of unity have been classified in
\cite{DGW}, while a non-constructible classification of irreducible $\vQnr_{\mathbb F}$-supermodules is obtained in \cite[Theorem 6.32]{Mo} by a generalised cellular structure\footnote{Green's codeterminant basis was a first such basis for the Schur algebra.} (\cite{GL}, \cite{DR0}).
Moreover, unlike the situation in \cite{DGW}, the link between representations of $\vQnr_{\mathbb F}$ and the quantum queer supergroup has not yet been established, since we do not know if the surjective map $\Phi_r$ given in Theorem 6.4 can be extended to the roots of unity case.
\end{rem}

\begin{appendix}
\section{Comparison with Mori's basis}
In \!\cite[Proposition\,6.10]{Mo}, Mori also obtained a basis for $x_\mu\HCR\cap\HCR x_\la$ for $\la,\mu\in\La(n,r)$, using a different approach via the combinatorics of circled tableaux. In terms of our notation, we may describe Mori's basis as follows. First, denote the basis elements given in Corollary \ref{trivial-intersect} by $m_\sfT$ where $\sfT$ runs over all circled row-standard tableaux. Then use a combinatorial relation between a circled row-semistandard tableau $\sfS$ and certain circled row-standard tableaux $\sfT$ to define $m_\sfS$ as a linear combination of those involved $m_\sfT$. We will see below that, if $\sfS$ gives the matrix pair $(A|B)\in M_n(\mathbb N|\mathbb Z_2)_{\mu,\la}$, then $m_\sfS=T_{A|B}$.  In other words, Mori's definition is simply write the element $T_{A|B}$ defined in \eqref{TAB} as a linear combination of the basis given in Corollary \ref{trivial-intersect} for $\sH_{r,R}^cx_\la$. We now establish the above described relationship.

A Young diagram of a composition $\la$ is defined by $Y(\la)=\{(i,j)| i\geq 1,1\leq j\leq\la_i\}$. A tableau of shape $\la$ is a function $\sfT: Y(\la)\rightarrow \{1,2,\ldots\}$.
The weight of a tableau $T$ is a composition $\mu=(\mu_1,\mu_2,\ldots)$ whose $i$-th component is $\mu_i=|\sfT^{-1}(i)|$.
$\sfT$ is said to be {\em row-semistandard} if the entries in each row of $\sfT$ are weakly increasing.
We denote by ${\rm Tab}_{\la;\mu}$ the set of row-semistandard tableaux of shape $\la$ and weight $\mu$.
It is known (cf. \cite[\S5.2]{Mo}) that the set ${\rm Tab}_{\la;\mu}$ is in bijection with the set $\mc D_{\mu,\la}$
and hence in bijection with the set $M_n(\N)_{\mu,\la}$. More precisely, given $\sfS\in {\rm Tab}_{\la;\mu}$,
let $M(\sfS)=(m_{ij})$ with $m_{ij}$ being the number $i$'s in the $j$-th row of $\sfS$.
Then clearly $M(\sfS)\in M_n(\N)_{\mu,\la}$ and by \eqref{mapj} we obtain a double coset representative $d_{\sfS}:=d_{M(\sfS)}\in \mc D_{\mu,\la}$.
Conversely, it is easy to see that the matrix $M\in M_n(\N)_{\mu,\la}$ uniquely determins a row-semistandard tableau $\sfS$.

For each $\sfS\in{\rm Tab}_{\la;\mu}$, let ${\rm Tab}_{\sfS}=\{\sfT\in{\rm Tab}_{\la;(1^r)}\mid \sfT_\mu=\sfS\}$, where $\sfT_\mu$ is the row-semistandard tableau obtained from $\sfT$ by replacing its entries $1,2,\ldots,\mu_1$ by $1$, $\mu_1+1,\ldots,\mu_1+\mu_2$ by $2$ and so forth. It is known from the discussion in \cite[Section 5.2]{Mo} that for each $\sfT\in{\rm Tab}_{\sfS}$, we have
\begin{equation}\label{dsfT}
d_{\sfT}=u d_{\sfS} \quad\text{and}\quad|\mc D^{-1}_{\nu(d_{\sfS}^{-1})}\cap \mf S_{\mu}|=|{\rm Tab}_{\sfS}|,
\end{equation}
where $u\in \mc D^{-1}_{\nu(d_{\sfS}^{-1})}\cap \mf S_{\mu}$.

\def\ci{\textcircled{$i$}}
\def\cone{\textcircled{1}}
\def\ctwo{\textcircled{2}}
\def\cthree{\textcircled{3}}
\def\cfive{\textcircled{5}}
\def\cfour{\textcircled{4}}
\def\csix{\textcircled{6}}
\def\cseven{\textcircled{7}}
\def\ceight{\textcircled{8}}
\def\yone{i_1}
\def\ytwo{i_2}
\def\ylast{i_k}
\def\zone{\textcircled{$i_1$}}
\def\ztwo{\textcircled{$i_2$}}
\def\zlast{\textcircled{$i_k$}}

Following \cite{Mo}, we introduce the notion of circled tableau. A circled tableau of shape $\la$ is a map $\sfS:Y(\la)\rightarrow \{1,2,\ldots\}\sqcup\{\textcircled{1},\textcircled{2},\ldots\}$. From a circled tableau $\sfS$ we obtain its underlying ordinary tableau $\sfS^\times$ by removing circles from numbers. The weight of a circled tableau is defined as that of underlying tableau. Denote by ${\rm Tab}_{\la;\mu}^c$ the set of circled tableaux $\sfS$ such that $\sfS^\times\in{\rm Tab}_{\la;\mu}$ and circled numbers must be placed at the rightmost of a bar $\young(ii\cdots i)$ in a row for every $i$.

For $\sfS\in{\rm Tab}_{\la;\mu}^c$, define $M^c(\sfS)=(A|B)\in M(\N|\Z_2)_{\mu,\la}$ by letting $B=(b_{ij})$ with $b_{ij}=0$ if the bar containing $i$ in the $j$-th row of $\sfS$ has the form $\young(ii\cdots i)$ and $b_{ij}=1$ if the bar containing $i$ in the $j$-th row of $\sfS$ has the form $\young(ii\cdots \ci)$, and $A=M(\sfS^\times)-B$. For example, for the circled tableau $\sfS=\young(1\cone2\cthree,15\cfive,\cfour)$,  we have $M^c(\sfS)=(A|B)$ with
 \begin{equation*}\label{S-AB}
M(\sfS^\times)=
\left(
  \begin{array}{ccccc}
   2 & 1& 0 & 0 & 0\\
   1 & 0& 0 & 0 & 0\\
   1 & 0& 0 & 0 & 0\\
   0 & 0& 1 & 0 & 0\\
   0 & 2& 0 & 0 & 0
  \end{array}
\right),\quad
 A=\left(
  \begin{array}{ccccc}
   1 & 1& 0 & 0 & 0\\
   1 & 0& 0 & 0 & 0\\
   0 & 0& 0 & 0 & 0\\
   0 & 0& 0 & 0 & 0\\
   0 & 1& 0 & 0 & 0
  \end{array}
\right),\quad
B=\left(
  \begin{array}{ccccc}
   1 & 0& 0 & 0 & 0\\
   0 & 0& 0 & 0 & 0\\
   1 & 0& 0 & 0 & 0\\
   0 & 0& 1 & 0 & 0\\
   0 & 1& 0 & 0 & 0
  \end{array}
\right).
 \end{equation*}
 Conversely, given $(A|B)\in M(\N|\Z_2)_{\mu,\la}$, we have $A+B\in M(\N)_{\mu,\la}$ and hence it uniquely determines a uncircled tableau in ${\rm Tab}_{\la;\mu}$ and the matrix $B$ determines the places of the circled numbers.
Hence, we obtain the following result.
\begin{lem}  For $\la,\mu\in\La(n,r)$, the map  $M^c: {\rm Tab}_{\la;\mu}^c\rightarrow M(\N|\Z_2)_{\mu,\la}$ is a bijection.
\end{lem}

In particular, we may use the set ${\rm Tab}_{\la;(1^r)}^c$ to label the basis elements for $\HCR x_\la$ given in Corollary \ref{trivial-intersect}
by setting, for $\sfT\in{\rm Tab}_{\la;(1^r)}^c$, 
\begin{equation}\label{msfT}
m_\sfT=T_{d_{\sfT^\times}}c^{\al_\sfT}x_\la\quad (\text{see \cite[\S6.4]{Mo}}),
\end{equation}
where $\al_\sfT=(\al_1,\al_2,\ldots,\al_r)\in\Z_2^r$ with $\al_k=1$ if the $k$-th entry is circled and $\al_k=0$ otherwise according to the top-to-bottom reading order.

By extending this definition to $\sfS\in{\rm Tab}_{\la;\mu}^c$, Mori further defined the element $m_\sfS\in x_\mu\HCR\cap\HCR x_\la$. Suppose the circled bars
in $\sfS$ are $\young(\yone\yone\cdots\zone),\young(\ytwo\ytwo\cdots\ztwo),\ldots,\young(\ylast\ylast\cdots\zlast)$ with lengths $r_1,r_2,\ldots,r_k$. By distributing
$$
\young(ii\cdots\ci )\mapsto \young(\ci i\cdots i)+q\young(i\ci\cdots i)+\cdots+q^{r-1}\young(ii\cdots\ci)
$$
for every $i=i_t$ and $r=r_t$, $1\leq t\leq k$, we first make a formal linear combintation $\sum_{0\leq l_t\leq r_t-1, 1\leq t\leq k}q^{l_1+l_2+\cdots+l_k}\sfR$ of circled row-semistandard tableaux $\sfR$,
where $\sfR=\sfR(l_1,l_2,\ldots,l_k)$ are the circled tableaux obtained from $\sfS$ by replacing the circled bars $\young(\yone\yone\cdots\zone)$, $\young(\ytwo\ytwo\cdots\ztwo),\ldots,\young(\ylast\ylast\cdots\zlast)$ by  $\young(\yone\cdots\zone\cdots\yone)$, $\young(\ytwo\cdots\ztwo\cdots\ytwo),\ldots,$ $\young(\ylast\cdots\zlast\cdots\ylast)$,
in which the circled
numbers are placed in the $(l_1+1)$-th,$(l_2+1)$-th,$\ldots$,$(l_k+1)$-th palces.
Then define
\begin{equation}\label{m_sfS}
m_{\sfS}=\sum_{0\leq l_t\leq r_t-1, 1\leq t\leq k}q^{l_1+l_2+\cdots+l_k}\sum_{\sfT\in\Gamma(\sfR)}m_\sfT,
\end{equation}
 where
$\Gamma(\sfR)$ consists of all $\sfT\in {\rm Tab}_{\la;(1^r)}^c$ such that $\sfT^\times\in{\rm Tab}_{\sfS^\times}$ and its positions of circles are same as that of $\sfR=\sfR(l_1,l_2,\ldots,l_k)$.

\begin{prop} For $\sfS\in{\rm Tab}_{\la;\mu}^c$, if $M^c(\sfS)=(A|B)$ then $m_{\sfS}=T_{A|B}.$
\end{prop}
\begin{proof} Observe that, for each $\sfR=\sfR(l_1,l_2,\ldots,l_k)$, $\Ga(\sfR)$ is a disjoint union of
$$
\Gamma(\sfR,\sfU):=\{\sfT\in{\rm Tab}^c_{\la;(1^r)}|\sfT^\times =\sfU, \sfT \text{ and }\sfR \text{ have the same positions of circles} \}
$$
for all $\sfU\in{\rm Tab}_{\sfS^\times}$. Thus, by \eqref{msfT},  \eqref{dsfT}, \eqref{cAB}, and \eqref{TAB},  \eqref{m_sfS} becomes
\begin{align*}
m_\sfS
&=\sum_{0\leq l_t\leq r_t-1, 1\leq t\leq k}q^{l_1+l_2+\cdots+l_k}\sum_{\sfU\in{\rm Tab}_{\sfS^\times}}\sum_{\sfT\in\Gamma(\sfR(l_1,l_2,\ldots,l_k),\sfU)}T_{d_{\sfT^\times}}c^{\al_\sfT}x_\la\\
&=\sum_{\sfU\in{\rm Tab}_{\sfS^\times}}T_{d_{\sfU}}\sum_{0\leq l_t\leq r_t-1, 1\leq t\leq k}q^{l_1+l_2+\cdots+l_k}
\sum_{\sfT\in\Gamma(\sfR(l_1,l_2,\ldots,l_k),\sfU)}c^{\al_\sfT}x_\la\\
&=\sum_{u\in \mc D^{-1}_{\nu(d_{\sfS^\times}^{-1})}\cap \mf S_{\mu}}T_{u}T_{d_{\sfS^\times}}c_{A|B}'x_\la\\
&=T_{A|B},
\end{align*}
as desired.
\end{proof}
\begin{rem} With this identification and by Proposition~\ref{hom-basis}, we now can remove the condition that $1+q$ is not a zero-divisor in $R$,  stated in \cite[Proposition~6.10]{Mo}. Of course, the standard assumption that the characteristic of $R$ is not equal to 2 in the super theory will always be maintained.
\end{rem}
\end{appendix}

\end{document}